\documentclass[11pt,letterpaper]{amsart}

\usepackage[margin=0.9in]{geometry}
\usepackage[utf8]{inputenc}
\usepackage[T1]{fontenc}

\usepackage{subfiles}
\usepackage{comment}
\usepackage{float}
\usepackage{marginnote}
\usepackage{euscript}
\usepackage[dvipsnames]{xcolor} 
\usepackage{graphicx}			
\usepackage{amssymb}
\usepackage{mathrsfs}
\usepackage{amsthm}
\usepackage{amsmath, amsrefs}
\usepackage[foot]{amsaddr}
\usepackage{mathtools}
\usepackage[cal=cm]{mathalfa}
\usepackage{stmaryrd}
\usepackage{upgreek}
\usepackage{bbm}
\usepackage[all]{xy}
\usepackage{thmtools}
\usepackage[hypertexnames=false, hyperfootnotes=false, colorlinks=true,linktocpage=true, allcolors=highlight, bookmarksdepth = 2]{hyperref}
\usepackage{cleveref}
\usepackage{url}
\usepackage{float}
\usepackage[full]{textcomp}
\makeatletter
\newcommand{\citecomment}[2][]{\citen{#2}#1\citevar}
\newcommand{\citeone}[1]{\citecomment{#1}}
\newcommand{\citetwo}[2][]{\citecomment[,~#1]{#2}}
\newcommand{\citevar}{\@ifnextchar\bgroup{;~\citeone}{\@ifnextchar[{;~\citetwo}{]}}}
\newcommand{\citefirst}{\@ifnextchar\bgroup{\citeone}{\@ifnextchar[{\citetwo}{]}}}

\makeatother

\usepackage{tikz,tikz-cd,tikz-3dplot}
\usepackage{pgfplots,pgfmath}
\pgfplotsset{compat=1.18}
\usetikzlibrary{calc}
\usetikzlibrary{fadings}
\usetikzlibrary{decorations.pathmorphing}
\usetikzlibrary{decorations.pathreplacing}
\usetikzlibrary{patterns}
\usetikzlibrary{arrows,shadows,positioning, calc, decorations.markings, 
	hobby,quotes,angles,decorations.pathreplacing,intersections,shapes}
\usepgflibrary{shapes.geometric}
\usetikzlibrary{fillbetween,backgrounds}

\usepackage{xcolor}

\definecolor{highlight}{HTML}{3465a4}

\usepackage{anyfontsize}
\usepackage[lining]{libertine}%
\usepackage{courier}
\usepackage{csquotes}
\makeatletter
\renewcommand*\libertine@figurestyle{LF}
\makeatother
\usepackage[libertine,libaltvw,liby]{newtxmath}
\usepackage{microtype}

\usepackage{array}
\newcolumntype{H}{>{\setbox0=\hbox\bgroup}c<{\egroup}@{}}

\BibSpec{article}{%
+{}{\sc \PrintAuthors} {author}
+{,}{ } {title}
+{,}{ \textit} {journal}
+{,}{ } {volume}
+{}{ \IfEmptyBibField{journal}{}{\PrintDate{date}}} {transition}
+{,}{ } {pages}
+{,}{ } {status}
+{.}{} {transition}
+{}{ \, Preprint available at \tt} {eprint}
+{.}{} {transition}
}

\BibSpec{book}{%
+{}{\sc \PrintAuthors} {author}
+{,}{ } {title}
+{.}{ \textit} {series}
+{,}{ } {volume}
+{.}{ } {publisher}
+{,}{ } {place}
+{,}{ } {date}
+{.}{} {transition}
}

\BibSpec{incollection}{
+{} {\sc \PrintAuthors} {author}
+{,} { } {title}
+{,} { in \textit} {booktitle}
+{,}{ } {series}
+{,}{ } {volume}
+{}{ \PrintDate } {date}
+{,} { pp.~} {pages}
+{,}{ } {publisher}
+{ }{ Preprint available at \tt} {eprint}
+{.}{} {transition}
}

\BibSpec{inproceedings}{
+{} {\sc \PrintAuthors} {author}
+{,} { } {title}
+{,} { in \textit} {booktitle}
+{,} { pp.~} {pages}
+{,} { \PrintDateB} {date}
+{,} { available at \eprint} {eprint}
+{.} {} {transition}
}

\BibSpec{collection.article}{
+{} {\sc \PrintAuthors} {author}
+{,} { } {title}
+{,} { in \it \PrintConference} {conference}
+{,} { pp.~} {pages}
+{.} {\PrintBook} {book}
+{,} { \PrintDateB} {date}
+{.} {} {transition}
}

\BibSpec{innerbook}{
+{.} { \emph} {title}
+{.} { } {part}
+{:} { \emph} {subtitle}
+{.} { } {series}
+{,} { } {volume}
+{.} { Edited by \PrintNameList} {editor}
+{.} { Translated by \PrintNameList}{translator}
+{.} { \PrintContributions} {contribution}
+{.} { } {publisher}
+{.} { } {organization}
+{,} { } {address}
+{,} { \PrintEdition} {edition}
+{,} { \PrintDateB} {date}
+{.} { } {note}
}

\tikzset{
	commutative diagrams/.cd, 
	arrow style=tikz, 
	diagrams={>=stealth}
}
\tikzset{
	arrow/.pic={\path[tips,every arrow/.try,->,>=#1] (0,0) -- +(0,4pt);},
	pics/arrow/.default={triangle 90}
}
\tikzset{->-/.style={decoration={
			markings,
			mark=at position .6 with {\arrow{latex}}},postaction={decorate}}
}
\tikzset{
	c/.style={every coordinate/.try}
}
\usepackage{enumerate}
\usepackage{enumitem}
\setlist[enumerate,1]{%
	label=(\roman*)	,itemsep=0.3em
	}
\setlist[itemize]{itemsep=0.3em}
\setlist[description]{itemsep=0.3em,labelindent=0.5cm}

\newcommand{\mf}{\mathfrak}

\DeclareMathOperator*{\Res}{Res}

\usepackage{xifthen}
\newcommand{\Gm}[1][]{%
	\ifthenelse{\isempty{#1}}%
	{\mathbb{C}^\times}
	{(\mathbb{C}^\times)^{#1}}
}

\newcommand{\rd}{\mathrm{d}}
\newcommand{\ri}{\mathtt{i}}
\newcommand{\re}{\mathrm{e}}

\newcommand{\de}{{\partial}}

\newcommand{\bbA}{\mathbb{A}}

\newcommand{\bbH}{\mathbb{H}}

\newcommand{\bbZ}{\mathbb{Z}}
\newcommand{\bbR}{\mathbb{R}}
\newcommand{\bbC}{\mathbb{C}}

\newcommand{\bbP}{\mathbb{P}}
\newcommand{\bbE}{\mathbb{E}}

\newcommand{\bbQ}{\mathbb{Q}}

\newcommand{\cR}{\mathcal{R}}
\newcommand{\cY}{\mathcal{Y}}
\newcommand{\cZ}{\mathcal{Z}}
\newcommand{\cO}{\mathcal{O}}

\newcommand{\cP}{\mathcal{P}}
\newcommand{\cC}{\mathcal{C}}

\newcommand{\cS}{\mathcal{S}}

\newcommand{\cW}{\mathcal{W}}

\newcommand{\cA}{\mathcal{A}}
\newcommand{\HH}{\mathcal{H}}
\newcommand{\cB}{\mathcal{B}}
\newcommand{\cF}{\mathcal{F}}

\newcommand{\RR}{\mathcal{R}}

\newcommand{\cM}{\mathcal M}

\renewcommand{\l}{\left}
\renewcommand{\r}{\right}

\def\beq{\begin{equation}}                     %
	\def\eeq{\end{equation}}                       %
\def\bea{\begin{eqnarray}}                     
	\def\eea{\end{eqnarray}}
\def\bary{\begin{array}} 
	\def\eary{\end{array}} 
\def\ben{\begin{enumerate}} 
	\def\een{\end{enumerate}}
\def\bit{\begin{itemize}} 
	\def\eit{\end{itemize}}
\def\nn{\nonumber}

\theoremstyle{plain}
\newtheorem{thm}{Theorem}[section]
\newtheorem*{thm*}{Theorem}
\newtheorem{lem}[thm]{Lemma}

\newtheorem{prop}[thm]{Proposition}
\newtheorem*{prop*}{Proposition}

\newtheorem*{conj*}{Conjecture}

\newtheorem{cor}[thm]{Corollary}
\newtheorem*{cor*}{Corollary}

\theoremstyle{definition}
\newtheorem{defn}[thm]{Definition}
\newtheorem{rmk}[thm]{Remark}

\theoremstyle{plain}

\theoremstyle{plain}

\theoremstyle{plain}

\theoremstyle{definition}

\theoremstyle{plain}

\crefname{equation}{Eq.}{Eqs.}
\crefname{eqnarray}{Eq.}{Eqs.}
\crefname{algo}{algorithm}{algorithms}
\crefname{conj}{conjecture}{conjectures}
\crefname{lem}{lemma}{lemmas}
\crefname{thm}{theorem}{theorems}
\crefname{claim}{claim}{claims}
\crefname{rmk}{remark}{remarks}
\crefname{prop}{proposition}{propositions}
\crefname{section}{section}{sections}
\crefname{appendix}{appendix}{appendices}
\crefname{cor}{corollary}{corollaries}
\crefname{figure}{figure}{figures}
\crefname{table}{table}{tables}
\crefname{example}{example}{examples}
\crefname{prob}{problem}{problems}
\crefname{assm}{assumption}{assumptions}
\crefname{defn}{definition}{definitions}
\crefname{notation}{notation}{notations}
\crefname{speculation}{speculation}{speculations}
\crefname{construction}{construction}{constructions}
\crefname{observation}{observation}{observations}
\crefname{innercustomthm}{Theorem}{Theorems}
\crefname{innercustomconj}{Conjecture}{Conjectures}
\crefname{innercustomassumption}{assumption}{Assumption}
\crefname{innerpracticalresult}{practical result}{practical results}

\newcommand{\bra}{\left\langle}
\newcommand{\ket}{\right\rangle}

\def\Xint#1{\mathchoice
   {\XXint\displaystyle\textstyle{#1}}%
   {\XXint\textstyle\scriptstyle{#1}}%
   {\XXint\scriptstyle\scriptscriptstyle{#1}}%
   {\XXint\scriptscriptstyle\scriptscriptstyle{#1}}%
   \!\int}
\def\XXint#1#2#3{{\setbox0=\hbox{$#1{#2#3}{\int}$}
     \vcenter{\hbox{$#2#3$}}\kern-.5\wd0}}

\def\dashint{\Xint-}

\crefname{equation}{Eq.}{Eqs.}
\crefname{eqnarray}{Eq.}{Eqs.}
\crefname{algo}{Algorithm}{Algorithms}
\crefname{conj}{Conjecture}{Conjectures}
\crefname{lem}{Lemma}{Lemmas}
\crefname{thm}{Theorem}{Theorems}
\crefname{customthm}{Theorem}{Theorems}
\crefname{claim}{Claim}{Claims}
\crefname{rmk}{Remark}{Remarks}
\crefname{prop}{Proposition}{Propositions}
\crefname{section}{Section}{Sections}
\crefname{appendix}{Appendix}{Appendices}
\crefname{cor}{Corollary}{Corollaries}
\crefname{figure}{Figure}{Figures}
\crefname{table}{Table}{Tables}
\crefname{example}{Example}{Examples}
\crefname{prob}{Problem}{Problems}
\crefname{assm}{Assumption}{Assumptions}
\crefname{defn}{Definition}{Definitions}
\crefname{customconj}{Conjecture}{Conjectures}

\numberwithin{equation}{section}

\title{Conifold gap  and all-genus mirror symmetry for local $\bbP^2$}
\author{Andrea Brini$^{1,2}$}
\address{$^1$ University of Sheffield, School of Mathematical and Physical Sciences, Hounsfield Road,  Sheffield S3 7RH, United Kingdom.}
\address{$^2$ On leave from CNRS, DR~13, Montpellier, France}
\email{a.brini@sheffield.ac.uk}

\begin{document}

\begin{abstract}
    
    The Conifold Gap Conjecture asserts that 
    the polar part of the Gromov--Witten potential of a Calabi--Yau threefold 
    near its conifold locus
    has a universal expression described by 
    the logarithm of 
    the Barnes $G$-function. 
    In this paper, I prove the Conifold Gap Conjecture 
    for the local projective plane. The proof relates the higher genus conifold Gromov--Witten generating series of local~$\bbP^2$
    to the thermodynamics of a certain statistical mechanical ensemble of repulsive particles on the positive half-line.
    As a corollary, this establishes the all-genus mirror principle for local $\bbP^2$ through the direct integration of the BCOV holomorphic anomaly equations.

\end{abstract}
\setcounter{tocdepth}{1}
	
	\maketitle

\section{Introduction}

This paper establishes the Conifold Gap Conjecture
\cite{Huang:2006si,Huang:2006hq} for local $\bbP^2$ -- the total space of the canonical bundle over the projective plane. In turn, this completes the proof of the all-genus mirror principle through the direct integration of the BCOV holomorphic anomaly equations for the first example of a non-trivial Calabi--Yau threefold. I explain somewhat informally the general context in \cref{sec:context}, and describe results and strategy of the proof in \cref{sec:localp2intro}.

\subsection{Context}
\label{sec:context}
Let $X$ be a non-singular Calabi--Yau threefold. The Gromov--Witten (GW) invariants of $X$ are a virtual dimension zero count
\[
N_{g,\beta}^X \coloneqq \int_{[\overline{M}_g(X, \beta)]^{\rm vir}} 1 \in \bbQ
\]
of genus-$g$ stable maps in $X$ with image representing a class $\beta \in H_2(X, \bbZ)$. Here $\overline{M}_g(X, \beta)$ denotes the
moduli space of genus-$g$, degree-$\beta$ unpointed stable maps to $X$, and $[\overline{M}_g(X, \beta)]^{\rm vir}$ its virtual fundamental class.

Completely determining all GW invariants of a smooth Calabi--Yau threefold is a notoriously hard problem, and a largely unsolved one in general. 
In genus zero, much is known since the  physics prediction \cite{COGP91:pairCYexactSolnSuperConfTh} 
and subsequent confirmation \cite{MR1653024,MR1621573}
of the rational GW invariants
of the quintic threefold via the mirror principle. Far less is known however in higher genus, despite spectacular progress in the study of target manifolds with additional structure, such as toric symmetry \cite{AKMV05:TopVert,MOOP11:GWDTtoric,Bouchard:2007ys,Fang:2016svw}. The state-of-the-art on the construction of an all-genus version of the mirror principle is given by the physics-motivated suggestion \cite{Klemm:1999gm,Huang:2006hq,Huang:2006si} of a ``direct integration'' approach  to the Holomorphic Anomaly Equations of  Bershadsky--Cecotti--Ooguri--Vafa (BCOV) \cite{Bershadsky:1993cx}, which ostensibly determine the fixed-ge\-nus generating functions of GW invariants up to a finite (albeit fast-growing with the genus) number of free parameters. For $X$ a Calabi--Yau threefold with Picard number one, this approach
can be 
distilled into postulating five conjectural statements: 
\ben
\item the Yamaguchi--Yau Finite Generation Property \cite{Yamaguchi:2004bt,Alim:2007qj}\,;
\item the BCOV Holomorphic Anomaly Equations \cite{Bershadsky:1993cx}\,; 
\item the Orbifold Regularity property \cite{Huang:2006hq}\,;
\item the Conifold Gap property  \cite{Huang:2006si}; and
\item the Castelnuovo bound  \cite{Huang:2006hq}\,. 
\een
In informal terms, and deferring precise statements to \cref{sec:localp2intro}, the Finite Generation Property~(i) asserts that the fixed-genus generating functions of the invariants $N^X_{g,\beta}$ belong to a certain finitely-generated polynomial ring $\RR_X$, determined by genus-zero data alone. The Holomorphic Anomaly Equations~(ii) recursively and fully determine, for each $g$, their dependence on a particular generator of $\RR_X$ (the BCOV propagator) in terms of lower genus data. Finally, the Orbifold Regularity, Conifold Gap, and Castelnuovo bound (iii)--(v) prescribe asymptotic conditions at boundary points in the stringy-extended K\"ahler moduli space of $X$ (respectively, at its orbifold, conifold, and large volume points), reducing the computation of genus-$g$ GW invariants to a finite-dimensional problem.
These speculative constraints are conjectured to determine the GW invariants $N^X_{g,\beta}$ up to high genus ($g \leq 53$ for the quintic threefold \cite{Huang:2006hq}) when $X$ is projective; and up  to {\it any} genus when $X$ is toric \cite{HKR:HAE}, in which case the Castelnuovo bound (v) is redundant. Recent spectacular advances using the Givental--Teleman reconstruction theorem \cite{MR1901075,MR2917177}
 have led to a mathematical proof of (i)--(iii) and (v) for some central examples of one-parameter Calabi--Yau threefolds, such as the quintic threefold \cite{Guo:2018mol, Chang:2021scw,Liu:2022agh,Alexandrov:2023zjb}
 and the local projective plane \cite{LP18:HAE,Coates:2018hms}.

The conjectural Conifold Gap property (iv) has on the other hand stood out as the single, and particularly elusive, exception in this list\footnote{The Conifold Gap has so far been known to (trivially) hold only for the resolved conifold $X=\mathrm{Tot}_{\bbP^1}(\cO(-1)\oplus \cO(-1))$, where it is an elementary consequence of the higher genus multi-covering formula for the super-rigid local curve \cite{Marino:1998pg, BP06:Rigidity,MR1728879}. In particular, in that case it follows from already knowing a closed expression for all GW invariants in the first place.}. It prescribes a universal expression for the polar part of the fixed-genus GW generating series of $X$ near the discriminant locus of the mirror in terms of the free energy of the one-dimensional bosonic string at  self-dual radius \cite{Distler:1990mt,Ghoshal:1995wm}, which is in turn encoded in the asymptotics of the logarithm of the Barnes $G$-function \cite{MR2723248}.
 The main difficulty towards a proof lies in the fact that the physical heuristics behind the conjecture is inherently \emph{non-geometrical} in character: unlike (i)--(ii), the Conifold Gap does not emerge as a general consequence of the R-matrix quantisation formalism for Gromov--Witten theory; and unlike (iii) and (v), it cannot be phrased as a geometric constraint on the invariants themselves, either directly as a bound on the existence of high-genus curves in $X$ for fixed degree, or indirectly through their relation via wall-crossing to FJRW/orbifold GW invariants of an auxiliary 
 target geometry.

\subsection{Local $\bbP^2$}
\label{sec:localp2intro}

I will now make precise the description of the BCOV axioms (i)--(iv) for the case when \[X=K_{\bbP^2}\coloneqq \mathrm{Tot}\,\cO_{\bbP^2}(-3)\] is the local projective plane\footnote{As $K_{\bbP^2}$ is toric, I shall omit a discussion of the Castelnuovo bound as no additional constraints arise from it in this case.}. Letting $H$ be the hyperplane class, we form the (large radius) Gromov--Witten formal generating series
\beq
\mathrm{GW}^{\rm LR}_g(Q) \coloneqq \sum_{d>0}Q^d N_{g, d H}^{K_{\bbP^2}}.
\label{eq:GWP2intro}
\eeq
In \eqref{eq:GWP2intro} we deliberately neglected the contribution of degree-zero stable maps to the genus-$g$ potentials, as the corresponding moduli spaces of stable maps are non-compact, and the associated GW invariants are not well-defined non-equivariantly. In terms of 
the hypergeometric function
\[
\mathsf{C}(y) \coloneqq
\, _2F_1\left(\frac{1}{3},\frac{2}{3};1;-27 y\right)
= \frac{\sqrt{3}}{2 \pi } \sum_{n \geq 0} \frac{\Gamma\l(\frac{1}{3}+n\r) \Gamma\l(\frac{2}{3}+n\r)}{(n!)^2}  (-27 y)^n \,,
\]
and the functions
\[
\mathsf{X}(y)\coloneqq \frac{1}{1+27 y}\,, \qquad \mathsf{F}(y) \coloneqq y \partial_y \log \mathsf{C}+\frac{1-\mathsf{X}}{3}\,,
\]
we consider, on the disc $|y|<1/27$, 
the analytic change of variables
\[
Q \longleftrightarrow y \longleftrightarrow \mathsf{q}_{\rm LR} 
\]
given  by
\[
Q(y) \coloneqq \exp 
\int_{+\infty}^y \mathsf{C}(y') \frac{\rd y'}{y'}=y-6 y^2+63 y^3+\cO\left(y^4\right)
\,,
\]
\[
-\mathsf{q}_{\rm LR}(y) \coloneqq \exp\int_{+\infty}^y \frac{\mathsf{X}(y')}{ \mathsf{C}(y')^2}\frac{\rd y'}{y'}=y-33 y^2+1035 y^3+\cO\left(y^4\right)
\,,
\]
so that 
\[
-\mathsf{q}_{\rm LR}(Q) = Q-9 Q^2+108 Q^3-1461 Q^4+\cO\left(Q^5\right)\,.
\]

%

With these definitions in place, the BCOV axioms (i)-(iv) for $X=K_{\bbP^2}$ can be described as follows.

The Finite Generation property (i) was proved in \cite{LP18:HAE,Coates:2018hms}. In its equivalent presentation given in \cite{Brini:2024gtn,BFGW21:HAE}, it asserts that $\mathrm{GW}_g^{\rm LR} \in \bbQ\llbracket Q\rrbracket$ has finite radius of convergence around $Q=0$ and, under the change-of-variables $Q = Q(y)$, it has a unique lift to the polynomial ring
\[
\RR_{K_{\bbP^2}} \coloneqq \bbQ[\mathsf{X}^{\pm 1}][\mathsf{F}]\,,
\]
with $\deg_{\mathsf{F}} \mathrm{GW}^{\rm LR}_g = 3g-3$. 

The Holomorphic Anomaly Equation (ii), also proved in 
\cite{LP18:HAE,Coates:2018hms}, is in turn the statement that, for $g \geq 2$,
\[
\frac{\partial \mathrm{GW}^{\rm LR}_g}{\partial \mathsf{F}} = \frac{3\mathsf{C}^2}{2\mathsf{X}}\l[ Q^2 \sum_{g'=1} ^{g-1} \frac{\partial \mathrm{GW}^{\rm LR}_{g-g'}}{\partial Q} \frac{\partial \mathrm{GW}^{\rm LR}_{g'}}{\partial Q} + \l( Q \frac{\partial}{\partial Q}\r)^2 \mathrm{GW}^{\rm LR}_{g-1}\r]\,,
\]
as an identity in $\RR_{K_{\bbP^2}}$,
determining recursively the coefficients of the positive degree terms of $\mathrm{GW}^{\rm LR}_g$ as a polynomial in $\mathsf{F}$ from the initial datum
\[
Q \de_{Q} \mathrm{GW}_{1} = -\frac{\mathsf{F}}{2 \mathsf{C}}-\frac{\mathsf{X}}{12 \mathsf{C}}\,.
\]
Under the formulation given in \cite{BFGW21:HAE,Brini:2024gtn}, the Orbifold Regularity property (iii) for $K_{\bbP^2}$, proved in \cite{MR3960668}, further constrains $\mathrm{GW}_g^{\rm LR}$ to belong to the finite-dimensional $\bbQ$-vector space
\[
\mathrm{GW}^{\rm LR}_g \in 
\mathrm{span}_\bbQ \bigg\{ \mathsf{X}^n\, \mathsf{F}^m \bigg\}_{\substack {0\leq m\leq 3g-3 \\ 1-g\leq n\leq 2g-2\\ 3n+m\geq 0}}
\,.
\]
Altogether, the properties (i)--(iii) completely fix $\mathrm{GW}^{\rm LR}_g$ from lower genus data up to its constant term in $\mathsf{F}$,
\beq
\mathrm{GW}^{\rm LR}_g\Big|_{\mathsf{F} \to 0} \in 
\mathrm{span}_\bbQ \Big\{ \mathsf{X}^n \Big\}_{n=0}^{2g-2}
\,.
\label{eq:modamb}
\eeq

The conjectural Conifold Gap property (iv) finally fixes the leftover $(2g-1)$-dimensional ambiguity at every genus up to an additive constant, corresponding to the non-equivariantly ill-defined constant map term which we deliberately discarded in the degree expansion \eqref{eq:GWP2intro}. To state it, we introduce on the open disc $|1+27 y|<1$
the analytic change of variables
\[
t_{\rm CF} \longleftrightarrow y \longleftrightarrow \mathsf{q}_{\rm CF} 
\]
defined by
\[
t_{\rm CF}(y) \coloneqq \frac{2\pi }{\sqrt{3}} \int_{+\infty} ^y \log\mathsf{q}_{\rm LR}(y') \mathsf{C}(y') \frac{\rd y'}{y'} = 1+27 y +\frac{11}{18} (1+27 y)^2
+
\cO(1+27y)^3\,,\]
\beq \mathsf{q}_{\rm CF}(y) \coloneqq \exp\l(\frac{4 \pi^2}{3}\frac{1}{\log \mathsf{q}_{\rm LR}}\r)=\frac{1}{27} (1+27 y)+\frac{5}{243} (1+27 y)^2
+ \cO(1+27y)^3
\label{eq:qfricke}
\eeq
so that
\[
\mathsf{q}_{\rm CF} = \frac{1}{27} t_{\rm CF} -\frac{1}{486} t_{\rm CF}^2+\frac{1}{13122} t_{\rm CF}^3 + \cO(t_{\rm CF}^4)\,.
\]
For $g \geq 2$, we define the higher genus \emph{conifold} GW potentials as
\beq
\mathrm{GW}_g^{\rm CF} \coloneqq 
\mathrm{GW}_g^{\rm LR}\Big|_{\mathsf{q}_{\rm LR} \to \mathsf{q}_{\rm CF}}\,.
\label{eq:GWCF}
\eeq

From \cite{Coates:2018hms}, the genus-$g$ conifold GW potential has a pole of order $2g-2$ at the conifold point $t_{\rm CF} = 1+27 y = q_{\rm CF} =0$. The Conifold Gap Conjecture \cite{Huang:2006si,Huang:2006hq,Ghoshal:1995wm,Coates:2018hms} predicts the full expression of the polar part of its Laurent expansion in the local coordinate $t_{\rm CF}$.
\begin{conj*}[Conifold Gap property for $K_{\bbP^2}$]
Let $g \in \bbZ_{\geq 2}$. Near $t_{\rm CF}=0$, the genus-$g$ conifold GW potential of $K_{\bbP^2}$ has the asymptotic behaviour
\beq
\mathrm{GW}_g^{\rm CF} = \frac{3^{g-1} B_{2g}}{2g(2g-2)} t_{\rm CF}^{2-2g}  +\cO(1)\,. 
\label{eq:cgintro}
\eeq
Here $B_n$ denotes the $n^{\rm th}$ Bernoulli number,
\[
\frac{t}{\re^{t}-1} \coloneqq \sum_{n \geq 0} \frac{B_n t^n}{n!}\,.
\]
\end{conj*}

Up to the factor of $3^{g-1}$, the leading coefficient in the Laurent expansion of $\mathrm{GW}_g^{\rm CF}$ near $t_{\rm CF}=0$ coincides, by the Harer--Zagier formula, with the orbifold Euler characteristic $\chi(\cM_g)$ of the moduli space of smooth genus-$g$ curves, while the vanishing of the subleading polar coefficients is the `conifold gap'. Imposing the asymptotic behaviour \eqref{eq:cgintro} gives $2g-2$
linearly independent constraints on the modular ambiguity in \eqref{eq:modamb}, fixing $\mathrm{GW}_g^{\rm LR}$ uniquely up to an additive constant\footnote{This constant is fixed by imposing that $\mathrm{GW}_g^{\rm LR}(Q)|_{Q=0}=0$ under our definition \eqref{eq:GWP2intro}, where degree zero invariants were discarded. Obviously, this is the single point where reinstating a constant map contribution to $\mathrm{GW}_g^{\rm LR}(Q)$ by considering a $\bbC^\times$-action on $K_{\bbP^2}$ 
would differ from the above, by fixing this constant ambiguity to a suitably adjusted value.}
. \\
Assuming the Conifold Gap property as a postulate, the authors of \cite{HKR:HAE} were able to effectively code the calculation of $\mathrm{GW}_g^{\rm LR}$ as far as computer power allowed them\footnote{Dispensing of all BCOV axioms, an all-genus version of local mirror symmetry for $K_{\bbP^2}$, and indeed for any toric Calabi--Yau threefold, is in principle available through the ``remodelling'' conjecture 
\cite{Bouchard:2007ys}, 
now a theorem \cite{Fang:2016svw,Eynard:2012nj}. However hardly any higher genus GW invariants can be computed directly in that approach, with results for toric local surfaces (such as $X=K_{\bbP^2}$) limited to $g=2$ \cite{Brini:2009nbd,MR2861610}.} ($g=105$ in 2008). The conjecture was subsequently checked up to genus $g=7$ with an explicit calculation of $\mathrm{GW}_g^{\rm CF}$ by Coates--Iritani in \cite{Coates:2018hms}. 

Our main result in this paper is the following:

\begin{thm*}[=Corollary~\ref{cor:cg}]
The Conifold Gap Conjecture holds for $X=K_{\bbP^2}$ in all genera.
\end{thm*}

\subsubsection{Strategy of the proof}

The leading term in the conifold expansion \eqref{eq:cgintro} is known to arise from the asymptotics of the Gaussian Unitary Ensemble at large rank. For $N \geq 1$, we write $\HH_N$ for the real vector space of $N\times N$ Hermitian matrices, and 
let \[\mathrm{vol}_{\rm HS}(U(N))\] denote the volume of the unitary group under the un-normalised Haar measure induced by the Hilbert--Schmidt inner product on $\mathfrak{u}(N)=\ri \HH_N$,
\[\bra X, Y\ket_{\rm HS} \coloneqq \mathrm{Tr}(X^\dagger Y)
\,.\]  
The GUE partition function is
\[
Z^{\rm GUE} \coloneqq \frac{1}{\mathrm{vol}_{\rm HS}(U(N))} \int_{\HH_N} \rd M\, \re^{-\frac{N}{2 t} \mathrm{Tr} M^2}\,,  
\]
where $t \in \bbR_{>0}$ is the so-called 't~Hooft coupling, and \[\rd M\coloneqq \prod_{i=1}^N \rd M_{ii} \prod_{i<j} \rd{\rm Re} M_{ij} \rd{\rm Im} M_{ij}\] is the Lebesgue measure on $\HH_N\simeq \bbR^{N^2}$. Using that
\[
\mathrm{vol}_{\rm HS}(U(N)) = \frac{(2\pi)^{\frac{N(N+1)}{2}}}{\prod_{k=1}^{N-1} k!}\,, \qquad \int_{\HH_N} \rd M\, \re^{-\frac{t}{2 N} \mathrm{Tr} M^2} = \l( \frac{2\pi t}{N}\r)^{N^2/2}\,,
\]
the Stirling approximation of the factorial and the Euler--Maclaurin formula give, for $N \to \infty$ and fixed $t>0$,
\beq
\log Z^{\rm GUE} \sim \frac{N^2}{2} \l(\log t-\frac{3}{2}\r) -\frac{1}{12} \log N + \zeta'(-1) + \sum_{g=2}^\infty \frac{B_{\rm 2g} t^{2-2g}}  {2g(2g-2)} \l(\frac{N}{t}\r)^{2-2g}\,.
\label{eq:GUEexp}
\eeq
The Conifold Gap property could follow from showing that the Laurent expansion of $\mathrm{GW}^{\rm CF}_g$ near $t_{\rm CF}=0$  coincides, up to regular terms in $t_{\rm CF}$, with the asymptotic expansion of the GUE in $t/N$ around $N=\infty$, upon identifying $t = \sqrt{3} t_{\rm CF}$. One potential scenario for this to happen is if the genus-$g$ conifold GW generating series of $K_{\bbP^2}$ \eqref{eq:GWCF} are shown to arise from the large rank asymptotics of a random matrix ensemble obtained as a real-analytic perturbation, around zero 't~Hooft coupling, of the GUE. 

This is indeed the route that will be followed in this paper. The proof will be structured through the following four main steps.
\bit
\item I will introduce a real-analytic perturbation of the GUE given in terms of a certain generalised Coulomb gas with confining and real-analytic convex potential, and a 2-body repulsion which is asymptotically of 2D-Coulomb type in the limit of small positive 't~Hooft coupling.
\item The model will be shown to rigorously admit an asymptotic expansion at $N\to \infty$ by applying results of \cite{Borot:2013pda,Borot:2013qs}, whose terms are determined recursively from the planar resolvent and connected planar 2-point function (or, equivalently, from a certain equivariant version thereof) through the Eynard--Orantin topological recursion \cite{Eynard:2007kz,Borot:2013lpa}.
\item The initial data in the topological recursion will be shown to solve uniquely certain non-local Riemann--Hilbert problems. The Rie\-mann--Hilbert problem for the planar resolvent had previously appeared in closely-related form in previous work on the exact solution of the six-vertex model on a random lattice \cite{Kostov:1999qx,Kazakov:1998ji,Dijkgraaf:2002dh,Zakany:2018dio}.
\item The solutions of the Riemann--Hilbert problems for the planar resolvent and connected 2-point function will then be shown to coincide with the disc and annulus  Eynard--Orantin differentials (in the conifold polarisation) associated to the Hori--Iqbal--Vafa mirror family of $K_{\bbP^2}$ \cite{Hori:2000ck}, i.e. the universal family over the modular curve $Y_1(3)$, up to a certain
automorphism 
of the 
spectral curve data.
\eit 
Invoking the homothety covariance properties of the closed sector of the topological recursion \cite{Eynard:2007kz,Hock:2025wlm} together with the proof \cite{Eynard:2012nj,Fang:2016svw} of the ``remodelling'' conjecture for $X=K_{\bbP^2}$ \cite{Bouchard:2007ys}, the above implies that the conifold Gro\-mov--Witten potentials equate, at every genus, to a perturbation of the coefficients of the large rank expansion of the GUE free energy starting at order $\cO(t_{\rm CF}^0)$ (\cref{thm:mm=gw}). The statement of the Conifold Gap Conjecture then follows via \eqref{eq:GUEexp}.

\subsubsection{Motivation and relation to other work} The central idea driving our 
approach 
is the heuristic expectation that Gro\-mov--Witten theory 
admits a dual holomorphic gauge theory description near conifold loci, akin -- and indeed mirror-symmetric -- to the topological gauge theories appearing in the Go\-pa\-ku\-mar--Ooguri--Vafa correspondence \cite{Gopakumar:1998ki,OV00:KnotInvTopStr}. The relevant gauge theory would be 
the holomorphic version of Chern--Simons theory studied in \cite{Witten:1992fb} to describe the space-time theory of the open topological B-model, dimensionally reduced onto the exceptional locus of a minimal resolution of the mirror Calabi--Yau geometry. The GUE is then physically expected to arise from its field-theoretic localisation, asymptotically near an ordinary double point, 
to a $0$-dimensional $U(N)$ gauge theory with Gaussian potential \cite{Dijkgraaf:2002fc}.

Much as this line of reasoning would constitute the royal way towards establishing the Conifold Gap, putting it on a mathematically rigorous footing would be substantially harder than proving the gap conjecture itself. On top of the foundational difficulty of rigorously defining $U(N)$ holomorphic Chern--Simons theory as a fully-fledged quantum theory, deriving its predicted relation at large $N$ to the genus expansion of Gromov--Witten theory appears to be a formidably difficult task. A potential shortcut would be available, however, if the open $B$-model at finite distance from the discriminant could be expressed in surrogate form through a zero-dimensional field theory,
which could be rigorously shown to provide a real-analytic deformation of the GUE and recover the Gromov--Witten generating series in the conifold polarisation via its large $N$ expansion. 

For $X=K_{\bbP^2}$, a suitable candidate for such a `conifold mean-field model' could be inferred from beautiful work of Mari\~no and Zakany \cite{Marino:2015ixa,Zakany:2018dio} on the spectral theory/topological strings correspondence \cite{Grassi:2014zfa}. In \cite{Grassi:2014zfa,Marino:2015ixa}, the authors give an integral representation of the spectral traces associated to the quantum mirror curve of a non-singular toric Calabi--Yau threefold $X$ with $H_4(X, \bbZ) \neq 0$. The $N \to \infty$ limit of these spectral traces should, furthermore, ostensibly describe the higher genus Gromov--Witten series of $X$ in the conifold polarisation.
Although both the 1-body potential and 2-body interaction are superficially quite different here, I believe that the asymptotics of the partition function of the model in this paper would reproduce that of the spectral trace integrals considered in \cite{Marino:2015ixa,Zakany:2018dio} to all orders in a $1/N$ expansion around $N=\infty$, up to correction terms of $\cO(\re^{-cN})$ with $c>0$. In particular, even though the respective finite-$N$ partition functions differ considerably, the Riemann--Hilbert problem for the planar resolvent considered here would reduce to a specialisation of the one studied in \cite{Marino:2015ixa}:
the specific potential and two-body interaction introduced in \cref{sec:cm} were precisely chosen to reverse-engineer the resulting saddle-point condition.
From a mathematical point of view, much of the virtue of the model introduced here stems from 
the manifest convexity properties required for the loop equation technology of \cite{Borot:2013pda,Borot:2013lpa} to work in the setup of a convergent generalised Coulomb gas. In particular, this allows to rigorously deduce that a unique equilibrium measure exists, it has connected support, and is off-critical \cite{Borot:2013pda}; that an asymptotic expansion of topological type exists \cite{Borot:2013qs}; and especially, that linear and quadratic loop equations rigorously hold \cite{Borot:2013lpa} and are solved by the Eynard--Orantin topological recursion: the latter provides the key conceptual bridge with the Givental--Teleman reconstruction of the higher genus generating series on the Gromov--Witten side. 

The analysis carried out in this paper would similarly apply to the convergent multi-species models that the work of \cite{Marino:2015ixa} would associate to general toric Calabi--Yau threefolds $X$ with $b_4(X)>0$. As for the case $X=K_{\bbP^2}$, the saddle-point equation for the quantum spectral traces could be read off from toric fan data, and suitable large $N$ mean-field models  can then be reconstructed from the resulting equation for the equilibrium density, and subjected to the same type of asymptotic study (at fixed filling fractions) that we carry out in \cref{sec:cm}. In particular, our study of $X=K_{\bbP^2}$ would generalise with minimal difficulty to the case of other local toric del Pezzo surfaces, such as $K_{\bbP^1 \times \bbP^1}$.

\subsubsection{Organisation of the paper} The rest of the paper is organised as follows. In \cref{sec:mf} I will give a quick overview of general Coulomb-type ensembles subject to convex and real-analytic interactions; I'll review the rigorous results known both on the structure of their large-$N$ expansion \cite{Borot:2013pda} and the loop-equations approach to the large-$N$ asymptotics via the Eynard--Orantin topological recursion \cite{Borot:2013lpa}. The specific mean-field ensemble of interest for this paper is introduced in \cref{sec:cm}, where it is proved that for sufficiently small positive 't Hooft coupling the model checks all the positivity conditions to ensure that the loop equations ideology of \cite{Borot:2013pda,Borot:2013lpa} rigorously applies. The large-$N$ expansion of the model is then shown to be solved by the closed sector of the Eynard--Orantin recursion applied to a particular family of elliptic spectral curves. These are then related in \cref{sec:lp2} to the 
local Calabi--Yau mirror geometry of $K_{\bbP^2}$ upon matching the planar resolvent with the Aganagic--Vafa disc amplitude for an outer Lagrangian brane in $K_{\bbP^2}$ at zero framing, up to an explicit homothety of the spectral data; the planar connected 2-point function with the corresponding annulus amplitude; and the Torelli marking of the random model with the vanishing cycle of the mirror family at the conifold point. The ``remodelled B-model'' formalism applied to $X=K_{\bbP^2}$ then entails that the Laurent expansion of $\mathrm{GW}_g^{\rm CF}$ near the conifold point returns 
(up to suitable scaling)
the genus-$g$ free energy of the GUE up to $\cO(1)$ in $t_{\rm CF}$, concluding the proof.

\subsection*{Acknowledgements} 
I was pushed to think about a proof of the Conifold Gap for $K_{\bbP^2}$ during the early stages of my collaboration 
with Y.~Sch\"uler \cite{Brini:2024gtn}. I am thankful to him for motivating me to get this started and -- quite a long time later -- to get it done. I am particularly grateful to the anonymous referees for their careful reading of the first version of the manuscript and their numerous suggestions, which greatly improved the quality of the exposition. I also benefited from helpful  correspondence and discussions with G.~Borot, K.~Iwaki, A.~Klemm, M.~Mari\~no, N.~Orantin and S.~Zakany. This work was supported by the EPSRC grants ref.~EP/S003657/2 and UKRI2394.

\section{Review of mean-field models at large $N$}
\label{sec:mf}
\subsection{Analytic strictly convex pairs}
In the following, for a measurable set $\mathsf{B} \subset \bbR^n$, we will write  $\cM_+(\mathsf{B})$ for the convex set of probability measures on $\mathsf{B}$, $\cM_0(\mathsf{B})$ for the set of signed measures on $\mathsf{B}$ with vanishing total mass, and $\bbE_{\mu}(f)$ for the expectation value of a $\mu$-measurable function on $\mathsf{B}$. 

Let $N\in \bbZ_{>0}$, $t\in \bbR_{>0}$, and $\mathsf{A} \subset \bbR$ be a connected open subset of the real line. Given 
$(V,L)$ with
$V\in \mathrm{C}^2(\mathsf{A})$, $\,  L\in \mathrm{C}^2(\mathsf{A}^2)$,
we will be interested in the absolutely continuous measure $\mu_{(V,L)}$ on $\mathsf{A}^N$
defined by
\beq
\rd \mu_{(V,L)} \coloneqq \frac{1}{(2\pi)^N N!}
\exp\l(-\frac{N}{t} \sum_{i=1}^N V(x_i)\r)\, \prod_{i<j} (x_i-x_j)^2 \prod_{i,j} \exp\l(L(x_i,x_j)\r) \rd^N x
\label{eq:measureVL}\,.
\eeq
We shall write
\[
Z^{(V,L)} \coloneqq \bbE_{\mu_{(V,L)}}(1) = 
\int_{\mathsf{A}^N} \rd\mu_{(V,L)}\,,
\]
\[
W_{k}^{{(V,L)}}(w_1, \dots, w_k) \coloneqq 
\frac{\de^k}{\de s_1 \dots \de s_k}
\log \bbE_{(V,L)}\l[\exp\l(\sum_{i=1}^k s_i \sum_{j=1}^N (w_i-x_j)^{-1}\r) \r]\Bigg|_{\mathbf{s}=0}\,,
\]
for, respectively, the total mass of \eqref{eq:measureVL},
and the Stieltjes transforms of its $k^{\rm th}$-order cumulants. For all $k>0$, the latter are holomorphic symmetric functions of $(w_1, \dots, w_k)\in (\widehat{\bbC} \setminus \mathsf{A})^k$ \cite{Borot:2013pda}.


\begin{defn}
\label{def:convex}
We say that the pair $(V,L)$ is \emph{strictly convex} if the following holds:
\begin{description}
\item[Strong confinement:] if $\mathsf{A}$ is unbounded, there exists $H : \mathsf{A} \to \bbR$ such that:
\bit
\item $2 t L(x,y)-V(x)-V(y) \leq  -H(x)-H(y)$ for $|x|+|y| \gg 1$;
\item $\liminf_{|x| \to \infty} \frac{H(x)}{2\log |x|} >t $; 
\end{itemize}
\item[1-body convexity:] $V$ has a single critical point $x^{\rm cr} \in \mathsf{A}$, which is a global minimum;
\item[2-body convexity:] 
for every $\nu \in \cM_0(\mathsf{A})$, the functional
\[
\bary{rcl}
Q : \cM_0(\mathsf{A})  & \longrightarrow & \bbR \\
\nu & \longrightarrow & \int_{\mathsf{A}^2} \big[\log|x-y|+L(x,y)\big] \rd\nu(x)\rd\nu(y) 
\eary
\]
is non-positive, and only vanishes for $\nu=0$.
\end{description}
If $(V,L)$ is strictly convex, the strong confinement property implies, in particular, that the measure $\mu_{(V,L)}$ has finite mass on $\mathsf{A}^N$:
\[
0<Z^{(V,L)} < + \infty\,.
\]
We will say that $(V,L)$ is an \emph{analytic strictly convex pair} (henceforth ASCP, in short) if, furthermore, both $V$ and $L$ are real-analytic functions of their respective arguments.
\end{defn}

\subsubsection{Equilibrium measures}

\begin{prop}[\cite{Borot:2013lpa,MR3265167,Borot:2013pda}]
Let $(V,L)$ be an ASCP. Then the \emph{energy functional} on $\cM_+(\mathsf{A})$,
\[
\bary{rcl}
\cF^{(V,L)} : \cM_+(\mathsf{A})  & \longrightarrow & \bbR\,,\\
\nu  & \longrightarrow & 
\int_\mathsf{A} V(x) \rd\nu(x)-t \int_{\mathsf{A}^2} \big[\log|x-y|+L(x,y)\big]\rd\nu(x)\rd\nu(y)\,,
\eary
\]
has a unique minimum $\nu_{(V,L)}$, with support on a compact interval
\[
\mathfrak{I}_0 \coloneqq [x_-,x_+] \coloneqq \overline{\mathrm{supp}(\nu_{(V,L)})} \subset \mathsf{A}\,.\]
Moreover, for all $t>0$, there exists $C_{(V,L)} \in \bbR$ such that, for all $x \in \mathsf{A}$,
\beq
V(x)-2 t\int_\mathsf{A} \big[ \log|x-y|+ L(x,y) \big] \rd\nu_{(V,L)}(y)  \geq C_{(V,L)}\,,
\label{eq:robin}
\eeq
with equality iff $x \in \mathfrak{I}_0$. 
\label{prop:minimum}
\end{prop}

When $L=0$, the existence and uniqueness of the minimising measure together with the bound \eqref{eq:robin} is a textbook result in potential theory (Gauss--Frostman lemma, see \cite[Ch.~1]{MR1485778}). Real-analyticity and convexity of $V$ further imply that $\nu_{(V,L)}$ is supported on a compact connected interval \cite{MR1327898}. The result holds verbatim for pairs $(V,L)$ with $L\neq 0$ upon assuming the ({\bf Strong confinement}) and ({\bf 2-body convexity}) parts of \cref{def:convex}: see
\cite{Borot:2013lpa,MR3265167,Borot:2013pda}.

\begin{prop}[\cite{Borot:2013pda}]
Let $(V,L)$ be an ASCP. The Radon--Ni\-ko\-dym derivative of $\nu_{(V,L)}$ w.r.t. the Lebesgue measure has the form  
\[
\rho_{(V,L)}(x) \coloneqq \frac{\rd \nu_{(V,L)}}{\rd x} =  \mathbf{1}_{\mathfrak{I}_0}(x) M_{(V,L)}(x,t)  \sqrt{(x-x_-)(x_+-x)}\,,
\]
where $\mathbf{1}_{\mathsf{B}}(x)$ denotes the characteristic function of $\mathsf{B} \subset \bbR$, 
$x_+(t)>x_-(t)$ are analytic in $t^{1/2}$ around $t=0$ with $x_+(0)=x_-(0)=x^{\rm cr}$,
and $M_{(V,L)}(x)$ is real-analytic and non-negative on $\mathfrak{I}_0$. 
\label{prop:sqrtcut}
\end{prop}
The statement is implied by \cite[Lem.~2.5]{Borot:2013pda}, all of whose hypotheses are satisfied by $(V,L)$ being an ASCP, once it is specialised to the case in which the support of the measure is compact and connected: this is ensured by \cref{prop:minimum}. The analyticity in $t^{1/2}$ follows, in turn, from Wigner's semi-circle law for the planar GUE \cite{Eynard:2015aea}. \\
\begin{prop}[Off-criticality]
For $t$ small enough, $M_{(V,L)}(x,t)$ is non-vani\-shing across $\mathfrak{I}_0^o= (x_-, x_+)$.
\label{prop:offcrit}
\end{prop}

\begin{proof}[Proof (sketch)]
As $V\in \mathsf{C}^2(\mathsf{A})$ and $x^{\rm cr}$ is a strict minimum by hypothesis, there exists an open interval $\mathsf{A}_{\rm cv} \subset \mathsf{A}$, with $x^{\rm cr} \in \mathsf{A}_{\rm cv}$, across which $V$ is strictly convex. As $x_\pm(t) \to x^{\rm cr}$ for $t \to 0$, we have that $\mathfrak{I}_0 \subset \mathsf{A}_{\rm cv}$ for $t$ small enough. \cref{prop:minimum,prop:sqrtcut} restricted to $\mathsf{A}_{\rm cv}$ then ensure that \cref{prop:sqrtcut} holds with $V$ a strictly convex function in a neighbourhood of $\mathfrak{I}_0$. This in turn implies (cf. \cite[Prop.~3.1]{MR1487983})  that $M_{(V,L)}(x,t)$ is strictly positive in the interior of $\mathfrak{I}_0$.
\end{proof}

\subsection{Asymptotic expansions at  $N=\infty$}

For $(V,L)$ an ASCP, the total mass and cumulant generating functions of $\mu_{(V,L)}$ admit an asymptotic expansion of topological type as $N \to \infty$. This follows from \cite[Cor.~3.4]{Borot:2013lpa} and \cite[Cor.~7.2 and 7.5]{Borot:2013pda}, whose assumed conditions \cite[Hyp.~3.2 and 3.5]{Borot:2013pda} are implied by the analyticity and strict convexity of $(V,L)$.

\begin{prop}[\cite{Borot:2013pda,Borot:2013lpa}] Let $(V,L)$ be an ASCP. The following asymptotic expansions hold as $N\to \infty$:
%
\begin{align}
\label{eq:logZVL}
\log Z^{(V,L)} \simeq &\, 
- \frac{1}{12} \log N +\sum_{g \geq 0} \l(\frac{N}{t}\r)^{2-2g} F_g^{(V,L)}(t)\,, \\
W_k^{(V,L)}(x_1, \dots, x_k) \simeq & \sum_{g \geq 0} \l(\frac{N}{t}\r)^{2-2g-k}\, W^{(V,L)}_{g,k}(x_1, \dots, x_k)\,,
\end{align}
with $W_{g,k}^{(V,L)}$ symmetric and  holomorphic for $(x_1, \dots, x_k) \in \big(\widehat{\bbC}\setminus (\mathfrak{I}_0 \cup \partial \mathsf{A})\big)^k$. In particular, in terms of the energy functional, we have
\[
F_0^{(V,L)} = \cF^{(V,L)}[\nu_{(V,L)}]\,.
\]
%
\label{prop:asympgen}
\end{prop}

Let $\alpha_{\rm cr} \coloneqq V''(x_{\rm cr})>0$. 
By \cite[Cor.~5.1 and Prop.~5.1]{Borot:2013qs}, upon rewriting
\[
\log Z^{(V,L)} = \log {Z^{(\frac{\alpha_{\rm cr} x^2}{2}/2,0)}}+\log \frac{Z^{(V,L)}}{Z^{(\frac{\alpha_{\rm cr} x^2}{2}/2,0)}}\,,
\]
the coefficients of the $1/N$-expansion of the second term are regular as $t\to 0^+$. In particular, for fixed $g \in \bbZ_{\geq 2}$, the polar part of the small $t$-asymptotics of $F_g$ is governed by the $1/N$ expansion of the free energy of the Gaussian Unitary Ensemble
\[
\log {Z^{(\frac{\alpha_{\rm cr} x^2
}{2},0)}} = \log 
G(N+1) -\frac{N}{2}\log (2\pi) +\frac{N^2}{2} \log \l(\frac{t}{\alpha_{\rm cr}N}\r)\,,
\]
where $G(z)$ denotes the Barnes~G-function, whose values over the positive integers are given by the superfactorial \[
G(N+1)= \prod_{j=1}^{N-1} j!\,.\]
From the known asymptotics of the Barnes $G$-function \cite{MR2723248},
\[
\log G(N+1) \sim \frac{N^2}{2} \l(\log N-\frac{3}{2}\r) -\frac{1}{12} \log N +\frac{N}{2}\log (2\pi)+\zeta'(-1)+\sum_{k=1}^{\infty} \frac{B_{2k+2}N^{-2k}}{4k(k+1)} 
\,,
\]
we deduce the following:
\begin{prop}[Gap property]
Let $(V,L)$ be an ASCP and $g\in \bbZ_{\geq 2}$. We have, near $t=0$, that
\[
F_g^{(V,L)}(t) = F_g^{(\frac{\alpha_{\rm cr} x^2}{2}
,0)}(t)+\cO(t^0) = \frac{\alpha_{\rm cr}^{2g-2}B_{2g}}{2g(2g-2)}\frac{1}{t^{2g-2}}+\cO(t^0)\,.
\]
\label{prop:HZ1}
\end{prop}

\subsection{The Eynard--Orantin recursion} \label{sec:EO} 
The main result of \cite{Borot:2013lpa} is a recursive reconstruction of the coefficients $F_g^{(V,L)}$ (resp. $W_{g,k}^{(V,L)}$) of the large $N$ expansion of
the partition function (resp. the connected multi-trace correlators), starting from the knowledge of the equilibrium measure and the planar 2-point function alone. 
Consider the branched double cover of the projective line
\bea
x\, :\, \bbP^1 & \longrightarrow & \bbP^1 \nn \\
z  & \longrightarrow & \frac{x_+-x_-}{4} \left(z+\frac{1}{z}\right)+\frac{x_++x_-}{2}\,.
\label{eq:ztox}
\eea
The map $z \mapsto x(z)$ is ramified at $z=\pm 1$, with branch points $x(\pm 1) = 
x_\pm$. 
For $D$ a sufficiently small (possibly disconnected) open neighbourhood of $\{-1,1\}$, and in terms of the $1/N$ expansion of the cumulants, we define meromorphic symmetric $k$-differentials 
$\omega^{(V,L)}_{g,k} \in \Gamma\l(D^k, (\cM_{D} \otimes_{\cO_D} \Omega^1_D)^{\boxtimes k}\r)^{S_k}$  as
\[
\omega^{(V,L)}_{g,k} \coloneqq \l( W^{(V,L)}_{g,k}(x(z_1), \dots, x(z_k)) + \frac{\delta_{g,0} \delta_{k,2}}{(x(z_1)-x(z_2))^2}\r) \rd x(z_1) \dots \rd x(z_k)\,,
\]
as well as the \emph{recursion kernel}
\[
K^{(V,L)}(z_0,z) \coloneqq \frac{1}{2}\l(\frac{\int_{z}^{1/z} \omega^{(V,L)}_{0,2}(z_0,\cdot )}{\omega^{(V,L)}_{0,1}(z)-\omega^{(V,L)}_{0,1}(1/z)}\r)\,. 
\]
In the following, for $a$ a letter and $I=(i_1, \dots, i_k) \in \bbZ^k$, we shall write $a_I$ as a shorthand for $(a_{i_1}, \dots, a_{i_k})$; and for $f(z)$ a meromorphic function on $\bbC \setminus \mathfrak{I}_0$, we will denote by $\oint_{\mathfrak{I}_0} f(z) \rd z$ its contour integral along a simple anti-clockwise oriented loop around $\mathfrak{I}_0$. 

\begin{prop}[Eynard--Orantin topological recursion, \cite{Eynard:2007hf,Borot:2013lpa}] Let $(V,L)$ be an ASCP. For $2g-2+k>0$, the symmetric $k$-differentials $\omega_{g,k}^{(V,L)}$ are single-valued and meromorphic on $D^k$, with only degree $2(3g-2+k)$ poles at $z_i=\pm 1$, and they are recursively determined as follows.
\ben
\item The planar resolvent, $(g,k)=(0,1)$, is the Stieltjes transform of the equilibrium measure,
\beq 
\frac{\omega^{(V,L)}_{0,1}}{\rd x} = 
W_{0,1}^{(V,L)}(x) = t \int_{x_-}^{x^+} \frac{\,\rho_{V,L}(y)}{x-y}\rd y\,. 
\label{eq:omega01}
\eeq
\item The planar two-point function, $(g,k)=(0,2)$, is the unique meromorphic symmetric 2-differential on $D^{2}$ which is regular away from the diagonal, and further satisfies  
\beq \omega^{(V,L)}_{0,2}(1/z_1,z_2) + \omega^{(V,L)}_{0,2}(z_1,z_2)+\frac{1}{\pi \ri} \oint_{\mathfrak{I}_0} \rd_{z_1}L(z_1, \zeta)  \omega^{(V,L)}_{0,2}(z_2,\zeta) =  \frac{\rd x(z_1)\, \rd x(z_2)}{(x(z_1)-x(z_2))^2}\,,
\label{eq:cutom02}
\eeq
\beq
\omega^{(V,L)}_{0,2}(z_1,z_2) = \frac{\rd z_1 \rd z_2}{(z_1-z_2)^2}+ \cO(1)\,, \qquad
\oint_{\mathfrak{I}_0} \omega^{(V,L)}_{0,2}(z_1,\cdot) =0\,.
\label{eq:omega02}
\eeq
\item For $2g-2+k>0$, we have
\begin{align}
\omega^{(V,L)}_{g,k+1}(z_0, z_1, \dots, z_k)= &
\sum_{z'\in \{\pm  1\}} \Res_{z=z'} K^{(V,L)}(z_0,z)\Bigg[ \omega^{(V,L)}_{g-1,k+2}(z, 1/z, z_1, \dots, z_k) \nn \\ 
& +\sum_{\substack{J \subseteq I, 0 \leq h \leq g\,,\\ (J,h) \neq (\emptyset,0),(I,g)}} \omega^{(V,L)}_{h,|J|+1}(z, z_J) \omega^{(V,L)}_{g-h,k-|J|+1}(1/z, z_{I \setminus J})\Bigg]\,.
\label{eq:omegagk}
\end{align} 
where $I=\llbracket 1, k \rrbracket$.
\item For $2g-2>0$, we have 
\[
\\
F_g^{(V,L)} =  \frac{1}{2-2g} \sum_{z'\in \{\pm  1\}} \Res_{z=z'}  \l(\int^z \omega^{(V,L)}_{0,1}\r)\omega^{(V,L)}_{g,1}(z) + c_g\,,
\]
with $c_g \in \bbC$ an additive constant independent of $t$. 
\een
\label{prop:EO}
\end{prop}
The statement follows from \cref{prop:sqrtcut,prop:offcrit} and from the existence, when $(V,L)$ is an ASCP 
\cite[Hyp.~3.3, 3.7--3.8]{Borot:2013lpa}, 
of Schwinger--Dyson equations\footnote{To be precise, the case $(g,k)=(1,1)$ requires a mild correction 
to account for the double poles of $\omega_{0,2}^{(V,L)}(z_1,z_2)$: the reader is referred to \cite[Sec.~3.7]{Borot:2013lpa} for the details. The expression for the quadratic loop equations given in \eqref{eq:QLEun}, however, covers this case as well.} for the convergent measure \eqref{eq:measureVL}:
\begin{align}
\label{eq:SD}
\cP^{(V,L)}_{g,k}(z;z_I)=\omega_{g-1,k+2}^{(V,L)}(z,z,z_I)+\sum_{J \subseteq I, 0 \leq h \leq g} \omega_{h, |J|+1}^{(V,L)}(z,z_J)\omega_{g-h, k+1-|J|}^{(V,L)}(z,z_{I \setminus J})
+ \\
\frac{1}{\pi \ri} \oint_{\mathfrak{I}_0} \rd_{z}L(z, \zeta) \bigg[ \omega^{(V,L)}_{g-1,k+2}(z,\zeta, z_I) 
+ \sum_{J \subseteq I, 0 \leq h \leq g} \omega_{h, |J|+1}^{(V,L)}(z,z_J)\omega_{g-h, k+1-|J|}^{(V,L)}(\zeta, z_{I \setminus J})
\bigg] \nn \\ - \rd V(x(z)) \omega_{g,k+1}^{(V,L)}(z,z_I)+ \sum_{i \in I} \frac{\omega_{g,k}^{(V,L)}(z,z_{I \setminus \{i\}})\rd x(z) \rd x(z_i)}{(x(z)-x(z_i))^2} \,, \nn
\end{align}
with $\cP^{(V,L)}_{g,k}(z;z_I)$ having double zeroes at $z=\pm 1$ as a quadratic differential in $z$. This in turn implies \cite[Cor.~3.4, Prop.~3.13--3.16]{Borot:2013lpa} that \eqref{eq:omega01}--\eqref{eq:omegagk} solve linear loop equations,
\begin{align}
0=
\omega_{g,k+1}^{(V,L)}(z,z_I)+\omega_{g,k+1}^{(V,L)}(1/z, z_I) + \frac{1}{\pi \ri} \oint_{\mathfrak{I}_0} \rd_{z}L(z, \zeta) \bigg[ \omega^{(V,L)}_{g,k+1}(z,\zeta, z_I) \nn \\
+ \sum_{J \subseteq I, 0 \leq h \leq g} \omega_{h, |J|+1}^{(V,L)}(z,z_J)\omega_{g-h, k+1-|J|}^{(V,L)}(\zeta, z_{I \setminus J})
\bigg]
\label{eq:LLEun}
\end{align}
as well as quadratic loop equations:
\beq
\widetilde{\cP}^{(V,L)}_{g,k}(z,;z_I) = 
\omega_{g-1,k+2}^{(V,L)}(z,1/z,z_I)+\sum_{J \subseteq I, 0 \leq h \leq g} \omega_{h, |J|+1}^{(V,L)}(z,z_J)\omega_{g-h, k+1-|J|}^{(V,L)}(1/z,z_{I \setminus J})\,,
\label{eq:QLEun}
\eeq 
where $\widetilde{\cP}^{(V,L)}_{g,k}(z;z_I)$ has double zeroes at $z=\pm 1$ as a holomorphic quadratic differential in $z$. Existence and uniqueness of a solution of \eqref{eq:cutom02}-\eqref{eq:omega02} in terms of the planar two-point function was established in \cite[Cor.~3.12]{Borot:2013lpa}. Finally, part (iv) of the statement is \cite[Prop~3.17]{Borot:2013lpa}.

Since $V(x)$ is $t$-independent and real-analytic on $\mathsf{A}$, 
we have for $I= \llbracket 1,k \rrbracket$ that
\[
\de_t W_k^{(V,L)}(x_I) = \frac{N}{2\pi\ri t^2} \oint_{\mathsf{A}} V(x) W_{k+1}^{(V,L)}(x,x_I) \rd x\,,
\]
\[
\frac{\de_t Z^{(V,L)}}{Z^{(V,L)}} = \frac{N}{2\pi\ri t^2} \oint_{\mathsf{A}} V(x) W_1^{(V,L)}(x) \rd x\,.
\]
Therefore, expanding in $1/N$, and arguing as in the proof of \cite[Prop~3.17]{Borot:2013lpa}, we find
\beq
\de_t \omega_{g,k}^{(V,L)}(z_I) = \frac{1}{2\pi\ri t^2} \oint_{\mathfrak{I}_0} V(x(z)) \omega_{g,k+1}^{(V,L)}(z,z_I)\,,
\label{eq:dtogk}
\eeq
\beq
\partial_t F_g^{(V,L)} = \frac{1}{2\pi\ri t^2}\oint_{\mathfrak{I}_0} V(x(z)) \omega_{g,1}^{(V,L)}(z) \,,
\label{eq:dtfg}
\eeq
where the derivative is taken at $x(z_i)$ fixed.
\subsubsection{Equivariant correlators}
\label{sec:equivcorr}
For later use we shall introduce an equivariant generalisation of the symmetric differentials $\omega_{g,k}^{(V,L)}$. These  carry equivalent information for the reconstruction of the asymptotic expansion \eqref{eq:logZVL} and, as we shall see, are often more amenable to a closed-form analytic representation.

For $d \in \bbZ_{>0}$, let $\varphi \coloneqq \re^{2 \pi \ri/d} \in \mu_d$ 
act multiplicatively on $x \in \bbC$,
\beq
\bary{rcl} 
\mu_d \times \bbC & \longrightarrow & \bbC \\ 
(\varphi^n,x) & \longrightarrow & \varphi^n x\,.
\label{eq:mudaction}
\eary
\eeq
In the following we shall assume that $\mathfrak{I}_0 \cap \varphi^n \mathfrak{I}_0 = \emptyset$ for all $n\neq 0$. 
For $\chi$ a non-zero element of the associated group ring over the complex numbers,
\[
0 \neq \chi \coloneqq \sum_{n=0}^{d-1} \chi_n \varphi^{n} \in \bbC[\mu_d]\,,
\]
 we define a symmetrised version of the correlators $W^{(V,L)}_{g,k}$ by summing over the orbits of \eqref{eq:mudaction} with weights specified by $\chi$ (see \cite{Brini:2011wi,Borot:2015fxa,Brini:2017gfi} for examples relevant to Chern--Simons theory) as:
\[
\cS_\chi W_{g,k}^{(V,L)}  \coloneqq
\sum_{n_1, \dots, n_k=0}^{d-1} 
\Big(\prod_{i=1}^k \chi_{n_i} \varphi^{n_i} \Big)
W^{(V,L)}_{g,k}(\varphi^{n_1} x_1, \dots, \varphi^{n_k} x_k) 
\,.\]
By \eqref{prop:asympgen}, $\cS_\chi W_{g,k}^{(V,L)}$ is  symmetric and holomorphic in its arguments away from the rotated images of the cut $\mathfrak{I}_0$ and the boundary $\partial \mathsf{A}$. 
For $z \in D$ and $i=0,\dots, d-1$, we will write $p \coloneqq (z,i) \in D_\chi$, where \[D_\chi \coloneqq \cup_{i=0}^{n-1} D = D \times \llbracket 0, d-1\rrbracket\,, \] 
and we will consider, as in \eqref{eq:ztox}, a map
\bea
x\, :\, D_\chi   & \longrightarrow & \bbP^1 \nn \\
(z, n)  & \longrightarrow & \varphi^n \l[ \frac{x_+-x_-}{4} \left(z+\frac{1}{z}\right)+\frac{x_++x_-}{2}\r]\,,
\label{eq:ztoxsymm}
\eea
which is a double cover over its co-domain with simple branching at the rotated images $\varphi^n x_\pm(t)$ of the endpoints of the cut $\mathfrak{I}_0$. We will write $\iota(p)$ for the deck involution sending $p=(z,n) \longrightarrow \iota(p) \coloneqq (1/z,n)$.
We then define
\[
\cS_\chi \omega_{g,k}^{(V,L)}  \coloneqq
\cS_\chi W_{g,k}^{(V,L)}(x(p_1), \dots, x(p_k))
\otimes_{i=1}^k\rd x(p_i) +
\frac{\delta_{g,0}\delta_{k,2}\rd x(p_1) \rd x(p_2)}{(x(p_1)-x(p_2))^2} 
\]
We will say that $\{\cS_\chi \omega_{g,k}^{(V,L)}\}_{g \geq 0, k \geq 1}$ satisfy quadratic loop equations if, for $2g+k>0$, the expression 
\[
\cS_\chi \omega_{g-1,k+2}^{(V,L)}(p,\iota(p),p_I)+\sum_{J \subseteq I, 0 \leq h \leq g} \cS_\chi\omega_{h, |J|+1}^{(V,L)}(p,p_J)\cS_\chi\omega_{g-h, k+1-|J|}^{(V,L)}(\iota(p),p_{I \setminus J})
\]
has double zeroes at $x(p)=\varphi^n x_\pm$ as a quadratic differential in $p$.
\begin{prop}
Suppose $\{\cS_\chi \omega_{g,k}^{(V,L)}\}_{g \geq 0, k \geq 1}$ satisfy quadratic loop equations. Then the Eynard--Orantin  recursion holds:
\begin{align}
\cS_\chi \omega^{(V,L)}_{g,k+1}(p_0, p_I)= &
\sum_{n=0}^{d-1} \sum_{\bullet \in \{+,-\}} \Res_{x(p)=\varphi^n x_\bullet} \cS_\chi K^{(V,L)}(p_0,p)\Bigg[ \cS_\chi \omega^{(V,L)}_{g-1,k+2}(p, \iota(p), p_I) \\ 
& +\sum_{\substack{J \subseteq I, 0 \leq h \leq g\,,\\ (J,h) \neq (\emptyset,0),(I,g)}} \cS_\chi \omega^{(V,L)}_{h,|J|+1}(p, p_J)\, \cS_\chi \omega^{(V,L)}_{g-h,k-|J|+1}(\iota(p), p_{I \setminus J})\Bigg]\,, 
\label{eq:EOsymm}
\end{align} 
where 
\[
\cS_\chi K^{(V,L)}(p_0,p) \coloneqq \frac{1}{2}\l(\frac{\int_{p}^{\iota(p)} \cS_\chi \omega^{(V,L)}_{0,2}(p_0,\cdot )}{\cS_\chi \omega^{(V,L)}_{0,1}(p)-\cS_\chi \omega^{(V,L)}_{0,1}(\iota(p))}\r)\,.
\]
\label{prop:EOsymm}
\end{prop}
The statement follows from the $\bbC$-linearity of the solvable linear loop equations in the arguments of the correlators, together with the assumption of quadratic loop equations (cf. \cite[Def.~2.12-2.13; Prop~2.7]{Borot:2013lpa}). 

Let us now {\it define}
\[
\cS_\chi \omega_{g,0}^{(V,L)} \coloneqq  \frac{1}{2-2g} \sum_{n=0}^{d-1} \sum_{\bullet \in \{+,-\}} \Res_{x(p)=\varphi^n x_\bullet}  \l(\int^p \cS_\chi \omega^{(V,L)}_{0,1}\r)\cS_\chi \omega^{(V,L)}_{g,1}(p)\,.
\]
 Picking $m \in \llbracket 0, d-1 \rrbracket$ with $\chi_{m} \neq 0$, the variational formula \eqref{eq:dtogk}, together with holomorphicity of $W^{(V,L)}_{0,2}$ away from $\mathfrak{I}_0$, implies that
\bea
\de_t \big[\cS_\chi W^{(V,L)}_{0,1}\big](x_1) &=& \frac{1}{2\pi\ri t^2} \sum_{n=0}^{d-1} \chi_{n} \varphi^{n-m} \oint_{\varphi^{m} \mathfrak{I}_0} V(\varphi^{-m} x_2) W^{(V,L)}_{0,2}(\varphi^{n} x_1, \varphi^{-m} x_{2}) \rd x_2 \nn \\
&=& \frac{1}{2 \pi \ri t^2 \chi_{m}} \oint_{\varphi^{m} \mathfrak{I}_0} V(\varphi^{-m} x_2) \cS_\chi W_{0,2}^{(V,L)}(x_1, x_2) \rd x_2\,. \nn
\eea
Therefore,
\[
\de_t \big[\cS_\chi \omega^{(V,L)}_{0,1}\big](p) = \oint_{\varphi^{m} \mathfrak{I}_0} \Lambda(x(q))\, \cS_\chi \omega^{(V,L)}_{0,2}(p,q) 
\]
with 
\[
\Lambda(x) \coloneqq \frac{V(\varphi^{-m} x)}{2\pi\ri \chi_{m} t^2}\,.
\]
Similarly, we find
\[
\de_t \big[\cS_\chi \omega_{0,2}\big](p_1,p_2) = \oint_{\varphi^{m} \mathfrak{I}_0} \Lambda(x(q)) \cS_\chi \omega_{0,3}(p_1,p_2,q).
\]
Then, by \cite[Thm.~5.1]{Eynard:2007kz} (see also \cite[Thm.~2.9,\,Cor.~2.10]{Borot:2013lpa}), \cref{prop:EOsymm} implies, for $I=\llbracket 1,k\rrbracket$, that
\[
\de_t \big[\cS_\chi \omega_{g,k}\big](p_I) = \oint_{\varphi^{m} \mathfrak{I}_0} \Lambda(x(q)) \cS_\chi \omega_{g,k+1}(p_I, q)
\]
for all $(g,k) \in (\bbZ_{\geq 0})^2$ with $2g-2+k>0$.
Setting $k=0$, comparing with \eqref{eq:dtfg}, and using again holomorphicity of $W_{g,1}^{(V,L)}(x)$ away from $\mathfrak{I}_0$, yields
\[
\de_t \big[\cS_\chi \omega^{(V,L)}_{g,0}\big] = \de_t F_g^{(V,L)}\,.
\]
We have proved the following statement.

\begin{prop}
Suppose $\{\cS_\chi \omega_{g,k}^{(V,L)}\}_{g \geq 0, k \geq 1}$ satisfy quadratic loop equations. Then, for $g \geq 2$, we have 
    \[
F_g^{(V,L)} =  \frac{1}{2-2g} \sum_{n=0}^{d-1} \sum_{\bullet \in \{+,-\}} \Res_{x(p)=\varphi^n x_\bullet}  \l(\int^p \cS_\chi \omega^{(V,L)}_{0,1}\r)\cS_\chi \omega^{(V,L)}_{g,1}(p) + d_g\,,
\]
with $d_g \in \bbC$ independent of $t$.
\label{prop:symmFg}
\end{prop}

\section{The conifold model}
\label{sec:cm}

We now specialise the discussion of the previous Section to a particular mean-field model. We fix 
$\mathsf{A}=\bbR_{>0}$ to be the open positive half-line, and define $\Phi \in C^\omega(\mathsf{A})$ as the solution of the following Cauchy problem on $\mathsf{A}$:
\[
 x^2 \Phi''(x) + x \Phi'(x) =   \frac{x}{(1-x+x^2)}\,, 
\]
\beq
 \Phi(1)=\Phi'(1)=0\,.
\label{eq:phi}
\eeq 
By construction, we have 
\beq 
\Phi(x)=\Phi(1/x)\,,
\label{eq:phiinv}
\eeq
with $\Phi(x)$ non-negative over its domain with a unique double zero at
$x=1$. More in detail, from \eqref{eq:phi}, we compute that
\beq
\Phi(x) =
\left\{\begin{array}{ll}
\frac{2 \pi \log(x)}{3 \sqrt{3}}+ \phi_\infty + \cO\l(\frac{1}{x}\r)\,, & x \to +\infty\,, \\
\\
\frac{(x-1)^2}{2}+ \cO(x-1)^3 \,, & x \to 1\,,
\end{array}\right.
\label{eq:phiasym}
\eeq
where
\[\phi_\infty = \frac{4 \pi^2}{27} -\frac{\psi ^{(1)}\left(\frac{1}{6}\right)+\psi ^{(1)}\left(\frac{1}{3}\right)}{18}=-1.17195\ldots \,, \]
and
$\psi^{(n)}(x)=\de_x^{n} \log \Gamma(x)$ is the $n^{\rm th}$ polygamma function.
\begin{figure}
\includegraphics[scale=0.45]{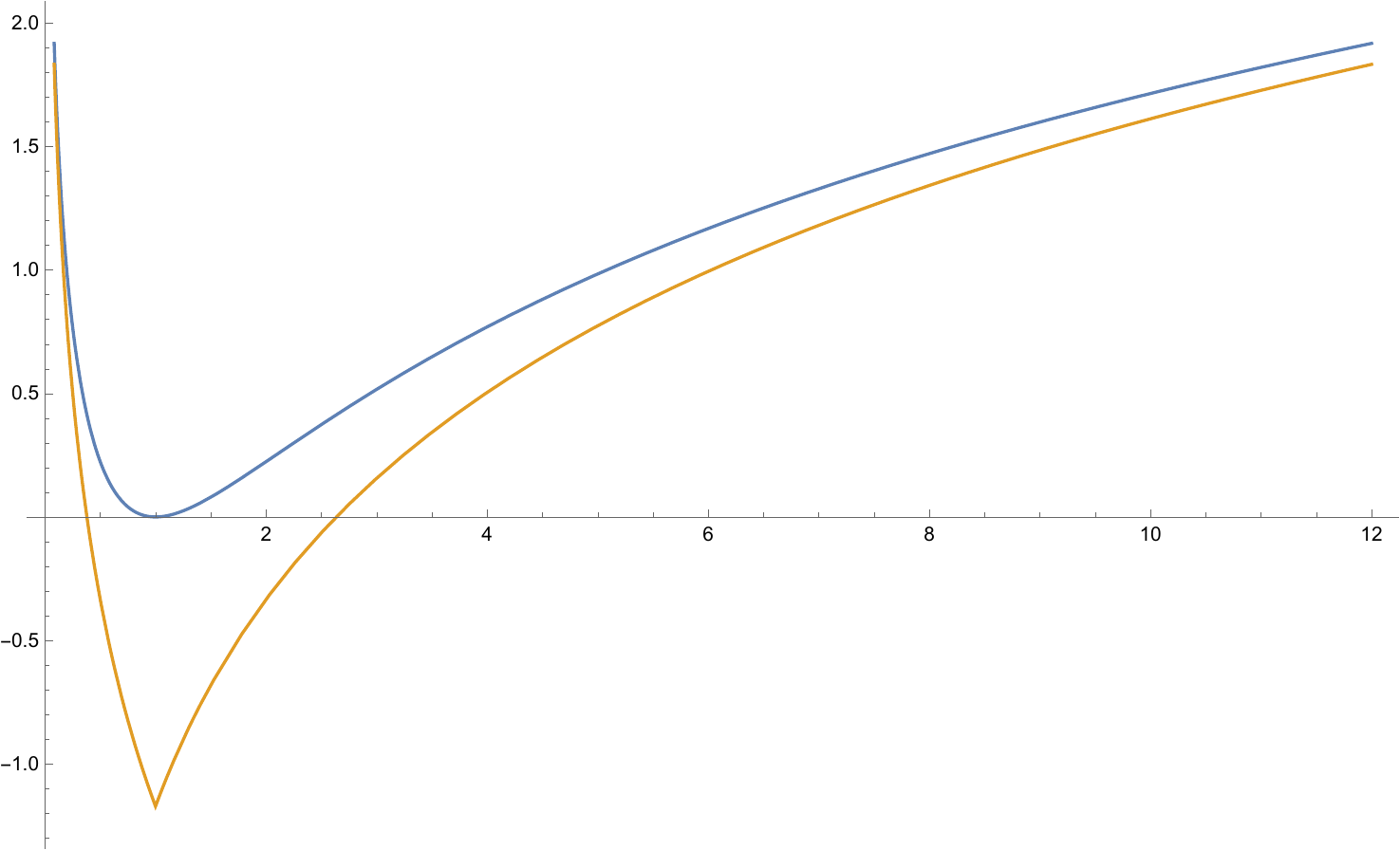}
\caption{The graph of $\Phi(x)$ (in blue) and of $\frac{2 \pi}{3\sqrt{3}} |\log x|+\phi_\infty$, superimposed (in orange).}
\end{figure}
Finally, we define 
$\Lambda \in \mathrm{C}^\omega(\mathsf{A}^2)$ as
\beq
\Lambda(x,y) \coloneqq -\frac{1}{2}\log \l(x^2 +x y+y^2\r)
\,.
\label{eq:lambda}
\eeq 

\subsection{Strict convexity}

\begin{lem}
For all $0<t<\pi/(3\sqrt{3})$ the pair $(\Phi, \Lambda)$ is an ASCP.
\label{lem:ASCP}
\end{lem}

\begin{proof}
Both $\Phi$ and $\Lambda$ are manifestly real-analytic on their respective domains; and we have already noted that
$\Phi$ has a unique global minimum on $\mathsf{A}$, satisfying ({\bf 1-body convexity}) with $\alpha_{\rm cr} \coloneqq \Phi''(1)=1$ from \eqref{eq:phiasym}. Since $\Lambda$ is strictly negative on $\mathsf{A}^2$, and 
$$
\lim_{x\to \infty} \frac{\Phi(x)}{2\log x }=\frac{\pi }{3\sqrt{3}}>t\,,$$
we see that the pair $(\Phi, \Lambda)$ satisfies ({\bf Strong confinement}) in the sense of \cref{def:convex} by taking $H(x)=\Phi(x)$\,.
Finally, let 
\beq 
\lambda(X) \coloneqq \frac{1}{2} 
\log \left(\frac{\left(\re^X-1\right)^2}{\re^{2 X}+\re^{X}+1}\right)\,,
\label{eq:lambdaz}
\eeq
so that, upon  setting $x=\re^X$ and $y=\re^Y$, we have
\[
\lambda(X-Y)= \log|x-y|+\Lambda(x,y)\,.
\]
Let then $\nu \in \cM_0(\mathsf{A})$ be a zero-mass measure on $\mathsf{A}$, and $\mu\in \cM_0(\bbR)$ the corresponding zero-mass measure on $\bbR$ such that $\exp_*\mu=\nu$.
Writing
\[
\widehat f (P) = \frac{1}{\sqrt{2 \pi}} \int_\bbR \re^{\ri P X} f(X) \rd X
\]
for the Fourier transform of $f \in L^2(\bbR)$, we have
\begin{align}
& 
\int_{\mathsf{A}^2} \big[\log|x-y|+\Lambda(x,y)\big] \rd\nu(x)\rd\nu(y) = 
\int_{\bbR^2} \lambda(X-Y) \rd\mu(X)\rd\mu(Y)
\nn \\
=&
\Big(\frac{\rd \mu}{\rd X}, \frac{\rd \mu}{\rd X} * \lambda \Big)_{L^2(\bbR)} = \Big(\widehat{\frac{\rd \mu}{\rd X}}, \widehat{\frac{\rd \mu}{\rd X}} \cdot \widehat \lambda \Big)_{L^2(\bbR)} = \int_{\bbR} \Big|\widehat{\frac{\rd \mu}{\rd X}}(P)\Big|^2 \widehat \lambda(P) \rd P\,.
\label{eq:2bodyLambda}
\end{align}
where $\frac{\rd \mu}{\rd X}$ denotes the Radon--Nikodym derivative of $\mu$ with respect to $X$. We compute that 
\[
\widehat{\lambda}(P) = 
-\frac{\sqrt{2 \pi }}{P} \frac{\sinh \left(\frac{2 \pi  P}{3}\right)}{2 \cosh \left(\frac{2 \pi  P}{3}\right)+1}\,,
\]
which is strictly negative for all $P\in \bbR$. Therefore, from \eqref{eq:2bodyLambda},
the integral
\[
\int_{\mathsf{A}^2} \big[\log|x-y|+\Lambda(x,y)\big] \rd\nu(x)\rd\nu(y)
\]
is non-positive, and only vanishing  when $\frac{\rd \mu}{\rd X}=0$ (and thus $\frac{\rd \nu}{\rd x}=0$) almost everywhere, thereby proving ({\bf 2-body convexity}).
\end{proof}

\begin{cor}
For $t\in (0,\pi/(3 \sqrt{3}))$, 
there is a unique minimising measure $\nu_{(\Phi,\Lambda)}$
for the energy functional $\cF^{(\Phi, \Lambda)}$, with Radon--Ni\-ko\-dym derivative
\[
\rho_{(\Phi,\Lambda)}(x) \coloneqq \frac{\rd \nu_{(\Phi,\Lambda)}}{\rd x} =  \mathbf{1}_{\mathfrak{I}_0}(x) M_{(\Phi,\Lambda)}(x) \sqrt{(x-x_-)(x_+-x)}\,,
\]
and $0<x_-=1/x_+<x_+$. Moreover, for $t \ll 1$, we have asymptotic expansions as $N\to\infty$
\begin{align*}
\log Z^{(\Phi,\Lambda)} \sim & \,- \frac{1}{12} \log N +\sum_{g \geq 0} \l(\frac{N}{t}\r)^{2-2g} F_g^{(\Phi,\Lambda)}\,, \nn \\
W_k^{(\Phi,\Lambda)}(x_1, \dots, x_k) \sim  & \sum_{g \geq 0} \l(\frac{N}{t}\r)^{2-2g-k}\, W^{(\Phi,\Lambda)}_{g,k}(x_1, \dots, x_k)\,,
\end{align*}
with
\[
W_{0,1}^{(\Phi,\Lambda)}(x)=t\,\int_{x_-}^{x^+} \frac{\rho_{\Phi,\Lambda}(y)}{x-y}\rd y\,,
\]
and \beq x\rho_{(\Phi,\Lambda)}(x)=\frac{1}{x}\rho_{(\Phi,\Lambda)}\Big(\frac{1}{x}\Big)\,, \qquad x W_{0,1}(x) = t- \frac{1}{x}W_{0,1}\l(\frac{1}{x}\r)\,.
\label{eq:rhoinv}
\eeq
\end{cor}
\begin{proof}
Uniqueness, connectedness of the support, and square-root behaviour near endpoints of the equilibrium measure follows from \cref{prop:sqrtcut}. Existence of a large rank asymptotic expansion for the free energy and correlators is ensured by \cref{prop:asympgen}, as is the expression for the planar resolvent in terms of the equilibrium measure. The fact that $x_-=1/x_+$ and the symmetry \eqref{eq:rhoinv} under the involution $x \to 1/x$ are a consequence, from \eqref{eq:phiinv} and $\eqref{eq:lambdaz}$, of the invariance of the energy functional in \cref{prop:minimum} under 
\[
\rd \nu(x) \longrightarrow \rd \nu(1/x)
\]
 which implies that its set of minima is invariant under the same involution. As the minimising measure $\nu_{(\Phi,\Lambda)}$ is unique, it is fixed by the involution, hence the symmetry \eqref{eq:rhoinv} of its density under $x \to 1/x$.

\end{proof}

\subsection{$\mu_3$-equivariant correlators}


Using the notation of \cref{sec:equivcorr}, we take 
\[
\varphi=\re^{2\pi\ri/3}\,, \qquad \chi = \sum_{n=0}^2 \chi_n \varphi^n \in \bbC[\mu_3]\,,
\]
with 
\[
\chi_0=0\,, \quad \chi_1 = 1\,, \quad \chi_2=-1\,.
\]
The corresponding $\chi$-equivariant correlators 
are then
%
%
%
\[
\cS_\chi \omega_{g,k}^{(\Phi,\Lambda)}\coloneqq 
\RR_{g,k}(x_1, \dots, x_k)
\frac{\rd x_1}{x_1}\dots \frac{\rd x_k}{x_k}\,,
\]
with
\[
\RR_{g,k}(x_1, \dots, x_k)
\coloneqq
\sum_{n_1, \dots, n_k \in \{1,2\}}
\l(\prod_{j=1}^k (-1)^{n_j+1}\varphi^{n_j} x_j \r)
W^{(\Phi,\Lambda)}_{g,k}(\varphi^{n_1} x_1, \dots, \varphi^{n_k} x_k)\,.
\]
For example, we have
\bea
\RR_{0,1}(x) &=& \,
\varphi x\, W^{(\Phi,\Lambda)}_{0,1}(\varphi x) - \varphi^2 x W^{(\Phi,\Lambda)}_{0,1}(\varphi^2 x)\,,
\label{eq:R01def}
\\
\RR_{0,2}(x_1, x_2) &=& x_1 x_2 \bigg[ \,
\varphi^2  \, W^{(\Phi,\Lambda)}_{0,2}(\varphi x_1, \varphi x_2) -   \, W^{(\Phi,\Lambda)}_{0,2}(\varphi x_1, \varphi^2 x_2) \nn \\ &-&   \, W^{(\Phi,\Lambda)}_{0,2}(\varphi^2 x_1, \varphi x_2)
+\varphi  \, W^{(\Phi,\Lambda)}_{0,2}(\varphi^2 x_1, \varphi^2 x_2)\bigg]
\,.
\label{eq:R02def}
\eea

\subsection{Riemann--Hilbert problem for the planar resolvent}
\label{sec:RHR01}
Writing $$\mathfrak{I}_\pm \coloneqq [\varphi^{\pm 1} x_-, \varphi^{\pm 1} x_+]\,,$$
 we see by \cref{prop:sqrtcut,prop:asympgen} that $\RR_{g,k}$ is holomorphic and single-valued on $(\widehat{\bbC} \setminus (\mathfrak{I}_+ \cup \mathfrak{I}_-))^k$, and has square-root branch cuts when $x_i \in \mathfrak{I}_\pm$, $i=1, \dots, k$ (see \cref{fig:cuts}). For $(g,k)=(0,1)$, the averages and discontinuities along the cuts are described by the following \nameCref{prop:RHPR01}.
\begin{prop}
\label{prop:RHPR01}
For all $x \in \mathfrak{I}_0$, we have
\beq
\label{eq:RHR01} 
\RR_{0,1}(\varphi x \re^{\pm \ri 0}) - 
\RR_{0,1}(\varphi^{2}x \re^{\mp \ri 0})  = -x \Phi'(x)\,, 
\eeq
\[
\RR_{0,1}(\varphi x \re^{+ \ri 0}) - 
\RR_{0,1}(\varphi x \re^{- \ri 0}) =\RR_{0,1}(\varphi^{2} x \re^{-\ri 0}) - 
\RR_{0,1}(\varphi^{2} x \re^{+ \ri 0})  = 2 \pi \ri x t \rho_{(\Phi, \Lambda)}(x)\,. 
\]
\end{prop}

\begin{proof}
    Taking the distributional $\log x$-derivative of \eqref{eq:robin} for $(V,L)=(\Phi, \Lambda)$, we get
\begin{align}
    x \Phi'(x) = & \dashint_{x_-}^{x_+} \frac{2 x t \rho_{(\Phi,\Lambda)}(y)}{x-y}\rd y +  \int_{x_-}^{x_+} 2 x t\de_x \Lambda(x,y)\rho_{(\Phi,\Lambda)}(y)\rd y\,,
\label{eq:saddlept}
\end{align}
where the dashed integral sign denotes the Cauchy principal value. 
By the Sokhotski--Plemelj lemma, we have, for $x\in \mathfrak{I}_0$, that
\begin{align}
W^{(\Phi,\Lambda)}_{0,1}(x \re^{+\ri 0})+ W^{(\Phi,\Lambda)}_{0,1}(x \re^{- \ri 0}) = & 2 t\, \dashint_{x_-}^{x_+} \frac{\rho_{(\Phi,\Lambda)}(y)}{x-y}\rd y \,, \nn \\ 
W^{(\Phi,\Lambda)}_{0,1}(x \re^{+\ri 0})- W^{(\Phi,\Lambda)}_{0,1}(x \re^{-\ri 0}) = & -2 \pi \ri t \rho_{(\Phi,\Lambda)}(x) \,, 
\label{eq:sp}
\end{align}
expressing, respectively, the average and the discontinuity of $W^{(\Phi,\Lambda)}_{0,1}$ on the cut $\mathfrak{I}_0$ in terms of the equilibrium density. From the discontinuity equation, we find that
\begin{align*}
\RR_{0,1}(\varphi x \re^{- \ri 0}) =\RR_{0,1}(\varphi^{2} x \re^{-\ri 0}) - 
\RR_{0,1}(\varphi^{2} x \re^{+ \ri 0})  = 
 2 \pi \ri x t \rho_{(\Phi,\Lambda)}(x)
\end{align*}
proving the second row of equalities in \cref{prop:RHPR01}. For the first equality, note that
%
%
\begin{align*}
    2 t \int_{x_-}^{x_+} x\de_x \Lambda(x,y)\rho_{(\Phi,\Lambda)}(y)\rd y
= &    
 -  t \int_{x_-}^{x_+} \frac{2x^2+x y}{x^2+xy+y^2}\rho_{(\Phi,\Lambda)}(y)\rd y \\
 =& -t \int_{x_-}^{x_+} \Bigg[
\frac{\varphi x}{\varphi x -y}+\frac{\varphi^2 x}{x\varphi ^2-y}
 \Bigg] \rho_{(\Phi,\Lambda)}(y)\rd y\,, \\
=& - \varphi x W^{(\Phi,\Lambda)}_{0,1}(\varphi x)- \varphi^2{x} W^{(\Phi,\Lambda)}_{0,1}\l(x\varphi^2\r)\,.
\end{align*}
Combining with \eqref{eq:sp}, we get
\begin{align*}
\RR_{0,1}(\varphi^{2}x \re^{\pm \ri 0} )-
\RR_{0,1}(\varphi x \re^{\mp \ri 0}) = &\,
x W^{(\Phi,\Lambda)}_{0,1}(x \re^{+\ri 0})+ x W^{(\Phi,\Lambda)}_{0,1}(x \re^{- \ri 0}) \nn \\ - &\, \varphi x W^{(\Phi,\Lambda)}_{0,1}(\varphi x)- \varphi^2{x} W^{(\Phi,\Lambda)}_{0,1}\l({\varphi^2}{x}\r)
= 
x \Phi'(x)\,,
\end{align*}
for all $x \in \mathfrak{I}_0$, concluding the proof.
\end{proof}
Consider now the derivative of \eqref{eq:RHR01} w.r.t. $\log x$,
\begin{align}
\varphi x\de_x \RR_{0,1}(\varphi x \re^{\pm \ri 0}) - 
\varphi^2 x\de_x \RR_{0,1}(\varphi^{2}x \re^{\mp \ri 0})= 
\frac{1}{\sqrt{3} \ri}\left(\frac{\varphi x}{1+{x}{\varphi }}-\frac{\varphi^2 x}{1+\varphi^2 x
   }\right)
\label{eq:RHR01der}
\end{align}
for all $x \in \mathfrak{I}_0$. We can homogenise \eqref{eq:RHR01der} in terms of the auxiliary function
\beq
\cW(x) \coloneqq x \de_x \RR_{0,1}(x)-\frac{1}{\sqrt{3} \ri}\l(\frac{x}{1+x}\r)\,,
\label{eq:Wdef}
\eeq
as
\beq
\cW(\varphi^{2}x \re^{\pm \ri 0} )-
\cW(\varphi x \re^{\mp \ri 0}) = 0\,, \qquad x \in \mathfrak{I}_0\,.  
\label{eq:Wspeq}
\eeq
Eq. \eqref{eq:Wspeq} 
is a non-local (homogeneous) Riemann--Hilbert problem for $\cW(x)$, relating its values on opposite sides of the cuts $\mathfrak{I}_\pm$. The 
complete Riemann--Hilbert problem for $\cW(x)$ and $\RR_{0,1}(x)$ can be spelled out as follows.
\begin{description}
\item[RHP(i):]
$\cW(x)$ is a single-valued meromorphic function on $\widehat{\bbC}\setminus (\mathfrak{I}_+ \cup \mathfrak{I}_-)$.
\item[RHP(ii):]
$\cW(x)$ satisfies the double periodicity relation \eqref{eq:Wspeq}, identifying its values on opposite sides of $\mathfrak{I}_\pm$ and $\mathfrak{I}_\mp$.
\item[RHP(iii):] $\cW(x) \rd \log x$ has a simple pole at $x=-1$ with residue $-1/(\sqrt{3} \ri)$; this follows from \eqref{eq:Wdef} and the fact that $\RR_{0,1}(x)$ is holomorphic across $\widehat{\bbC} \setminus (\mathfrak{I}_+ \cup \mathfrak{I}_-)$, by \cref{prop:asympgen}.
\item[RHP(iv):] $\cW(x)$ has an inverse-square root divergence at $x=\varphi x_\pm$ and $x=\varphi^2 x_\pm$, by \cref{prop:sqrtcut}.
\item[RHP(v):] $\cW(x)$ and $\RR_{0,1}(x)$ have a simple zero at $x=0$, by \eqref{eq:R01def}.
\item[RHP(vi):] $\cW(x)$ is regular at $x=\infty$, with $\cW(x) = -1/(\sqrt{3} \ri) + \cO(1/x)$. This follows from $W^{(\Phi,\Lambda)}_{0,1}(x)=\cO(1/x)$ by definition of the planar resolvent, $\RR_{0,1}(x) = \cO(1/x)$ by \eqref{eq:R01def}, and \eqref{eq:Wdef}.
\item[RHP(vii):] if 
$\cC$ is a simple clockwise-oriented small loop around $\mathfrak{I}_-$, we have  
\beq
\frac{\ri}{2\pi}
\oint_{\cC} \RR_{0,1}(x) \frac{\rd x}{x} =  t\,.
\label{eq:AMMper}
\eeq
by \eqref{eq:sp}.
\end{description}
Consider the topological cylinder $\bbR \times S^1$ obtained by cutting open the cuts $\mf{I}_\pm$ in the plane. Conditions (i)-(ii) together 
imply that $\cW$ can be pulled-back to a single-valued continuous function on the genus one surface $T$ obtained by closing up the cylinder $\bbR \times S^1$ through a continuous identification of the opposite boundary circles (see \cref{fig:glueing}). In terms of the universal cover, we have a diagram
\beq
\xymatrix@R=2.5cm@C=2.5cm{
& \mathbb{C} \ar[d]^x \ar@/^.75pc/[dl]^{\pi} \\
T \ar@/^.75pc/[ur]^{u} \ar[r] & \mathbb{R} \times S^1 
}
\label{eq:diagT}
\eeq
with 
\bit
\item $T \simeq \bbC/M$, with $M=\bra 2\omega_1 , 2\omega_2 \ket_\bbZ$ an integral lattice, with $2 \omega_i \in M$ primitive lattice elements for $i=1,2$;
\item $\pi : \bbC \to T$  the quotient map to $T$, with $u$ a section thereof;
\item $x: \bbC \to \bbR \times S^1 \simeq \bbC/\bra 2\omega_1 \ket$ the covering map to $\bbR \times S^1$, with the images of the edges of the fundamental domain of $M$ mapping to the base circle and vertical segment of the cylinder.
\eit
If, furthermore, the vertical arrow in the diagram can be constructed from a $2\omega_1$-periodic meromorphic surjection $$x: \bbC \longrightarrow \widehat{\bbC}\,,$$ then, under the natural complex structure on $T$ induced by $M$, the pull-back of $\cW$ via $x$ and $u$ will be single-valued and meromorphic on the elliptic curve $T \simeq \bbC/M$, and thus expressible in terms of elliptic functions.


\begin{figure}
    \centering

\begin{tikzpicture}[scale=1.2, thick]

\draw[] (0,0) circle (1.5);

\shade[ball color=gray!30] (0,0) circle (1.5);

\foreach \y in {0.5,-0.5} {
  \draw[fill=gray!1, dashed, opacity=0.95] (-0.3,\y) ellipse (0.3 and 0.15);
}
\node at (-0.9,0.5) {$\mathfrak{I}_+$};
\node at (-0.9,-0.5) {$\mathfrak{I}_-$};

\draw[-{Latex[length=3mm]}, thick] (-0.5,-0.7) arc (-45:-315:1);

\shade[ball color=gray!30] (5,0) circle (1.5);
\draw[](5,0) circle (1.5); 

\begin{scope}
  \shade[gray!30, opacity=.8, shading angle=135] 
    (4.8,0.5) to[out=120,in=90] (2.6,0.4) 
             to[out=-90,in=180] (4.7,-0.55) --
    (4.5,-0.25) to[out=180,in=-90] (2.9,0.4) 
              to[out=90,in=120] (4.3,0.4) -- cycle;
\end{scope}

\draw[thin] (4.8,0.5) to[out=120,in=90] (2.6,0.4) 
             to[out=-90,in=180] (4.6,-0.55);
\draw[thin] (4.3,0.4) to[out=120,in=90] (2.9,0.4) 
             to[out=-90,in=180] (4.5,-0.25);

  \draw[fill=gray!15, opacity=.5, dashed] (4.5,0.5) ellipse (0.3 and 0.15);

    \draw[fill=gray!25, opacity=.3, dashed] (4.5,-0.4) ellipse (0.3 and 0.15);

\end{tikzpicture}

    \caption{The cylinder given by the Riemann sphere with the sides of the intervals $\mathfrak{I_{\pm}}$ cut open (left), and the topological genus one surface obtained upon identifying their opposite sides (right).}
    \label{fig:glueing}
\end{figure}
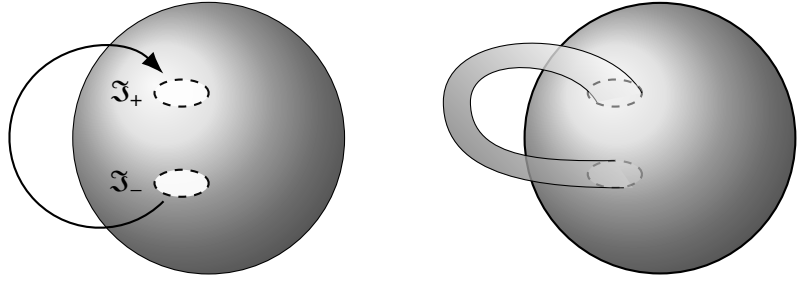

\subsubsection{Elliptic parametrisation} A general strategy to construct such a holomorphic cover was found in \cite{Kostov:1999qx} (see also \cite{Kazakov:1998ji, Dijkgraaf:2002dh} and especially \cite{Zakany:2018dio} for the case at hand). Let $\Delta$ be the fundamental domain centred around the origin for the translation action of $M$ on $\bbC$ (\cref{fig:cuts}), where without loss of generality we set the half-periods to be equal to
\[
    \omega_1=\frac{1}{2}\,, \qquad \omega_2= \frac{\tau}{2} \in \bbH\,, \qquad \omega_3 \coloneqq -\omega_1-\omega_2 = - \frac{1+\tau}{2}\,,
\]
and we further write
\[
\mathsf{q} \coloneqq \re^{2\ri \pi \tau}\,.
\]
We shall henceforth restrict to $\tau \in \ri \bbR^+$ (i.e. $0<\mathsf{q}<1$), so that $M$ is a rectangular lattice.  In the following, we will fix the section $u$ in \eqref{eq:diagT} as the map from $T$ into the fundamental region $\Delta$, where  as above
\[
\Delta \coloneqq \l\{u \in \bbC\,\Big|\, -\frac{1}{2}\leq \mathrm{Re}(u) < \frac{1}{2}\,,\,\, -\frac{|\tau|}{2}< \mathrm{Im}(u) \leq \frac{|\tau|}{2}\r\}\,,
\]
and we will write $\mathsf{p}_q$ for the point in $T$ with coordinate $q \in \Delta$,
\[
u(\mathsf{p}_q)=q\,.
\]
We also define, for $0\leq \epsilon \ll 1$,
\[
\Delta^{[\epsilon]}_\pm  \coloneqq  \l\{u \in \Delta \Big|\, \mathrm{Re}(\pm u) \geq 0\,,~~ \epsilon < |\mathrm{Re}(u)| <  \frac{1}{2}-\epsilon \,,~~ |\mathrm{Im}(u)|< \frac{|\tau|}{2}-\epsilon \r\}\,, 
\]
\[
\Delta_\pm  \coloneqq \Delta^{[0]}_\pm\,,
\]
and write $\cA$ for the horizontal segment $\mathrm{Im}(u)=-\frac{\tau}{2}+\ri\epsilon$, oriented towards increasing values of $\mathrm{Re}(u)$, and $\cB$ for the vertical edge $\mathrm{Re}(u)=\frac{1}{2}$, oriented towards increasing values of $\mathrm{Im}(u)$ (see \cref{fig:cuts}). These descend to, respectively, a longitudinal and meridian circle on $T$, and we will denote their homology classes generating $H_1(T, \bbZ)$ as
\[
\alpha \coloneqq [\cA]\,, \qquad \beta \coloneqq [\cB]\,.
\]
Consider now the meromorphic map
\begin{align}
    x : \bbC   
\, & \longrightarrow  \widehat{\bbC}\,, \nn \\
u\, & \longrightarrow  x(u)=-\frac{\vartheta_1(u+1/6)}{\vartheta_1(u-1/6)}\,,
\label{eq:xu}
\end{align} 
where
\beq
\vartheta_1(u) \coloneqq 2\mathsf{q}^{1/8} \sum_{k=0}^{\infty}(-1)^k \mathsf{q}^{k (k+1)/2} \sin (\pi  (2 k+1) u)
\label{eq:thetadef}
\eeq
is the first Jacobi Theta function. 
From the well-known quasi-periodicity properties of $\vartheta_1(u)$ 
under translation by $M$,
\beq
    \vartheta_1(u+1)=  -\vartheta_1(u)=\vartheta_1(-u)\,, 
    \qquad
\vartheta_1(u+\tau)=  -\frac{\re^{-2\pi \ri u}}{\mathsf{q}^{1/2}}\vartheta_1(u)\,,
\label{eq:thetaper}
\eeq
we see that $x(u)$ is a meromorphic function on $\bbC$ with only simple poles at $M+\frac{1}{6}$, and further satisfying
\begin{align}
& x(u)=x(u+1)\,, \qquad x(u+\tau)=\varphi^2 x(u)\,, \qquad x(u) = \frac{1}{x(-u)}\,.
\label{eq:xuprop}
\end{align}
In particular $x(u)$ is periodic under $u \to u+1$ and quasi-periodic under $u \to u+\tau$, with $\log x(u)$ odd in $u$ for the appropriate choice of branch of the logarithm.  From \eqref{eq:xu}, some special values are:
\[
x(-1/6)=0\,, \quad x(0)=-1\,, \quad x(1/6) = \infty\,, \quad x(\pm 1/2) = 1\,.
\]
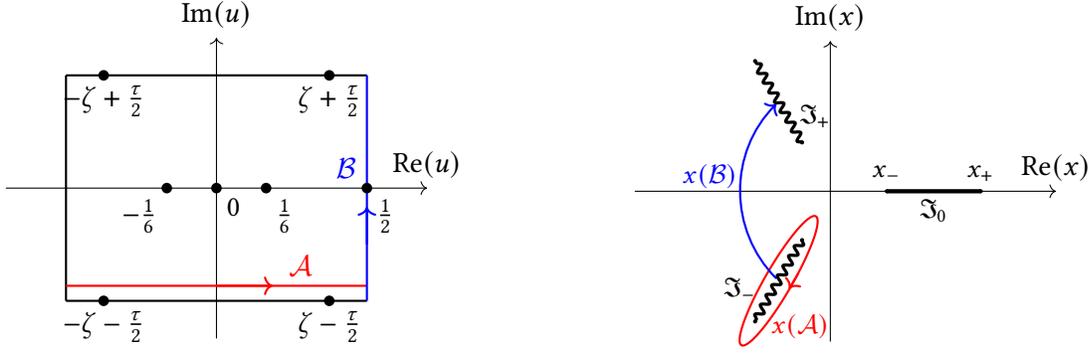
\begin{figure}

\begin{minipage}{.48\textwidth}

\begin{tikzpicture}[scale=1]
  \draw[->] (-2.8, 0) -- (2.8, 0) node[above] {$\mathrm{Re}(u)$};
  \draw[->] (0, -2) -- (0, 2) node[above] {$\mathrm{Im}(u)$};

  \draw[thick] (2, 1.5) -- (-2, 1.5);
  \draw[thick] (-2, 1.5) -- (-2, -1.5);
  \draw[thick] (-2, -1.5) -- (2, -1.5);
  
  \draw[thick, red] (-2, -1.3) -- (2, -1.3) node[above, xshift=-25pt] {$\cA$};
\draw[->, red, thick] (0, -1.3) -- (.75, -1.3);

  \draw[thick, blue] (2, -1.5) -- (2, 1.5) node[left, yshift=-35pt] {$\cB$};

\draw[->, blue, thick] (2, -.75) -- (2, -.25);

  \fill (0, 0) circle (2pt) node[below right] {$0$};

  \fill (2,0) circle (2pt) node[below right] {$\frac{1}{2}$};


  \fill (-.66,0) circle (2pt) node[below left] {$-\frac{1}{6}$};

  \fill (.66,0) circle (2pt) node[below right] {$\frac{1}{6}$};


    \fill (1.5,-1.5) circle (2pt) node[below] {$\mathfrak{u}-\frac{\tau}{2}$};

    \fill (-1.5,-1.5) circle (2pt) node[below] {$-\mathfrak{u}-\frac{\tau}{2}$};

        \fill (1.5,1.5) circle (2pt) node[below] {$\mathfrak{u}+\frac{\tau}{2}$};

    \fill (-1.5,1.5) circle (2pt) node[below] {$-\mathfrak{u}+\frac{\tau}{2}$};

\end{tikzpicture}

\end{minipage}
\begin{minipage}{.35\textwidth}
\end{minipage}
\begin{minipage}{.48\textwidth}
\begin{tikzpicture}[scale=1]
  \draw[->] (-2.6, 0) -- (3, 0) node[above] {$\mathrm{Re}(x)$};
  \draw[->] (0, -2) -- (0, 2) node[above] {$\mathrm{Im}(x)$};

  \pgfmathsetmacro{\rstart}{3/4}
  \pgfmathsetmacro{\rend}{2}
  \pgfmathsetmacro{\angleone}{120}
  \pgfmathsetmacro{\angletwo}{240}

  \pgfmathsetmacro{\xA}{\rstart*cos(\angleone)}
  \pgfmathsetmacro{\yA}{\rstart*sin(\angleone)}
  \pgfmathsetmacro{\xB}{\rend*cos(\angleone)}
  \pgfmathsetmacro{\yB}{\rend*sin(\angleone)}

    \pgfmathsetmacro{\xC}{\rstart*cos(\angletwo)}
  \pgfmathsetmacro{\yC}{\rstart*sin(\angletwo)}
  \pgfmathsetmacro{\xD}{\rend*cos(\angletwo)}
  \pgfmathsetmacro{\yD}{\rend*sin(\angletwo)}

  \pgfmathsetmacro{\xm}{(\xC + \xD)/2}
  \pgfmathsetmacro{\ym}{(\yC + \yD)/2}

  \pgfmathsetmacro{\vx}{\xD - \xC}
  \pgfmathsetmacro{\vy}{\yD - \yC}
  \pgfmathsetmacro{\theta}{atan2(\vy,\vx)}

  \pgfmathsetmacro{\a}{0.55*(sqrt(\vx*\vx + \vy*\vy)) * 1.2}
  \pgfmathsetmacro{\b}{0.1*(sqrt(\vx*\vx + \vy*\vy)) * 1.2}

  \draw[rotate around={\theta:(\xm,\ym)}, thick, red, domain=0:360, samples=200, variable=\t]
    plot ({\xm + 1.2*\a*cos(\t)}, {\ym + 1.2*\b*sin(\t)});

  \draw[->, thick, red] (-0.43, -1.1) -- ++(-0.12, -0.2);

  \pgfmathsetmacro{\xE}{(\xA + \xC)/2}
  \pgfmathsetmacro{\yE}{(\yA + \yC)/2}


   \draw[thick, blue] (\xC/2+\xD/2-0.03,\yC/2+\yD/2+0.03) to[out=135, in=225] (\xA/2+\xB/2-0.03,\yA/2+\yB/2-0.03);

  \draw[->, thick, blue] (-1.04, 0.7) -- ++(0.1, 0.2);

  
  \draw[ultra thick
  ] (\rstart, 0) -- (\rend, 0) ;
  \draw[very thick, decorate, decoration={snake, amplitude=.5mm,segment length=1.5mm}] (\xA,\yA) -- (\xB,\yB);
  \draw[very thick, decorate, decoration={snake, amplitude=.5mm,segment length=1.5mm}] 
    ({\rstart*cos(\angletwo)}, {\rstart*sin(\angletwo)}) -- 
    ({\rend*cos(\angletwo)}, {\rend*sin(\angletwo)});

  
    \filldraw[] (\rstart, 0) circle (0.02) node[above] {\small $x_-$};
    \filldraw[] ({\rstart*cos(\angleone)}, {\rstart*sin(\angleone)}) circle (0.02);
    \filldraw[] ({\rstart*cos(\angletwo)}, {\rstart*sin(\angletwo)}) circle (0.02);

    \filldraw[] (\rend, 0) circle (0.02) node[above] {\small $x_+$};
    \filldraw[] ({\rend*cos(\angleone)}, {\rend*sin(\angleone)}) circle (0.02);
    \filldraw[] ({\rend*cos(\angletwo)}, {\rend*sin(\angletwo)}) circle (0.02);

 \node at (1.375,-0.25) {\small $\mathfrak{I}_0$};

  \node at (-1.2,-1.3) {\small $\mathfrak{I}_-$};

    \node at (-0.4,-1.8) {\small
    \color{red} $x(\cA)$};
    \node at (-1.6,0.2) {\small \color{blue} $x(\cB)$};
    \node at (-0.2,1) {\small  $\mathfrak{I}_+$};
 
\end{tikzpicture}
\end{minipage}

\caption{The fundamental domain $\Delta$ of $T\simeq \bbC/M$ (left) and the complex plane with cuts $\mathfrak{I}_\pm$ for $\RR_{0,1}$ (right).
.}
\label{fig:cuts}
\end{figure}
The next Lemma further characterises the image under $x(u)$ of the boundary of $\Delta_{\pm}$ in $\widehat{\bbC}\setminus \mathfrak{I}_\pm$ and, in particular, the image of the fundamental cycles $\cA$ and $\cB$. We define, for $a \in \Delta$,
\[
    \phi_\pm(u) \coloneqq \log x\l(u \pm \frac{\tau}{2}\r) 
    \,,
    \qquad
    \chi_+(u) \coloneqq \log x\l( \tau u + \frac{1}{2}\r)\,,   \qquad  \chi_-(u) \coloneqq \log x\l(-\tau u \r)\,,
    \]
where in all cases we take the branch cut of the logarithm to run along the \emph{negative} real axis. By definition, restricting to $u \in (-1/2,1/2]$, we see that
\[
\exp\phi_-\l((-1/2,1/2]\r) = x(\cA)\,, \qquad \exp \chi_+\l((-1/2,1/2]\r) = x(\cB)\,.
\]

\begin{lem}
For $0<\mathsf{q} \ll 1$ and $v\in (-1/2,1/2]$, there are real analytic functions $\mathfrak{x}=\mathfrak{x}(\mathsf{q})>0$, $\mathfrak{u}=\mathfrak{u}(\mathsf{q})$ such that $\mathrm{Re}(\phi_\pm(v))$ has exactly two critical points at $v=\pm \mathfrak{u}$, and 
\[
\mathrm{Re}(\phi_\pm(v)) = [-\mathfrak{x}, \mathfrak{x}]\,, \qquad 
\mathrm{Im}(\phi_\pm(v)) = \mathrm{Im}(\chi_+(\pm 1/2)) = \mathrm{Im}(\chi_+(\mp 1/2)) =\pm\frac{2\pi}{3}\,, 
\]
\[
\mathrm{Re}(\chi_\pm(v)) = 0\,, \quad  \mathrm{Im}(\chi_\pm'(v)) <0\,,  \quad
 \frac{2 \pi}{3} \leq |\mathrm{Im}(\chi_\pm(v))| \leq \pi\,. 
\]

\label{lem:phichi}
\end{lem}

\begin{proof}
The proof hinges on a judicious application of the reality properties of elliptic functions for rectangular lattices. We omit the details at this point, but include them in full for the reader's convenience in \cref{sec:phichi}.
\end{proof}
\begin{prop}
 There exists an invertible map $t = t(\mathsf{q})$ 
 near $t=\mathsf{q}=0$ such that
\bit
\item the restriction of $x(u)$ to ${[-1/2 \pm \tau/2, 1/2 \pm \tau/2)}$ gives a branched double cover of the cut $\mathfrak{I}_\pm$; 
\item $x(\cA)$ is a simple clockwise oriented loop 
around the cut $\mathfrak{I}_-$\,;
\item $x(\cB)$ is a simple clockwise oriented arc of circumference between $\varphi^{-1}$ and $\varphi$\,.
\eit
\label{prop:CtoT}
\end{prop}

\begin{proof}
The change-of-variable in question is achieved by imposing that the end-points of the cut agree:
\beq
\log x_\pm(t) = \pm \mathfrak{x}(\mathsf{q})\,.
\label{eq:xplfrx}
\eeq
By \cref{prop:sqrtcut}, and since $x_+=1/x_-$, there exists a function $\kappa(t)$, real-analytic  in $t^{1/2}$ near $t=0$, such that
\[
\log x_+(t) = - 
\log x_-(t) = t^{1/2}\, \kappa(t)\,.
\]
In particular, we must have 
$\kappa(0) =2 \neq 0$ from the Gaussian limit to Wigner's semi-circle law for the equilibrium distribution. 
On the other hand, from the uniformly convergent $\mathsf{q}^{1/2}$-series expression of $\phi_\pm(u)$ in \eqref{eq:thetadef}, there exists a function $\mu(\mathsf{q})$, analytic in $\mathsf{q}^{1/2}$ for $\mathsf{q}\ll 1$, such that
\[
\mathfrak{x} = \mathrm{Re}(\phi_\pm(\mathfrak{u})) = \mathsf{q}^{1/2}\,\mu(\mathsf{q}) \,,
\]
where, at the first few orders we find
\beq
\mu(\mathsf{q}) = \sqrt{3} \l[
2+3 \mathsf{q}+\frac{27 \mathsf{q}^2}{20}+\frac{341 \mathsf{q}^3}{56}+\cO\left(\mathsf{q}^{7/2}\right)
\r]
\label{eq:muq}
\eeq
so that, in particular, $\mu(0)= 2\sqrt{3}\neq 0$. Then, by Lagrange--Buermann inversion, \emph{imposing} the identification
\eqref{eq:xplfrx} induces an invertible map $t = t(\mathsf{q})$, real-analytic\footnote{We will prove that this map is actually analytic in $\mathsf{q}$, rather than $\mathsf{q}^{1/2}$, in \cref{rmk:tq}.} 
in 
 $\mathsf{q}^{1/2}$,
 under which the critical values of $\exp \phi_\pm(v)$ are mapped to the endpoints of the cuts $\mathfrak{I}_\pm$  by \cref{lem:phichi}.
Equivalently, we have
\[
\varphi^{\pm 1} x_+ =
 x\l(\mathfrak{u}\pm \frac{\tau}{2}\r) \,, \qquad \varphi^{\pm 1} x_-  \coloneqq
 x\l(-\mathfrak{u}\pm \frac{\tau}{2}\r) \,,
\]
and where furthermore, by \cref{lem:phichi}, $x(\tau v+\frac{1}{2})=\exp(\chi_+(v))$ maps bijectively $v \in [-1/2,1/2)$ onto the arc of circumference $x(\cB)$ centred at the origin between $\varphi^{-1}$ and $\varphi$, oriented towards the latter.
As far as $\cA$ is concerned, by \cref{lem:phichi}, $\exp \phi_\pm(v)$ realises the interval $(-1/2,1/2]$ as a double-cover of the cuts $\mathfrak{I}_\pm$, branched at $\varphi^{\pm 1} x_+$ and $\varphi^{\pm 1} x_-$. Shifting the argument $v$ towards the interior of the fundamental domain by $\mp \ri \epsilon$ ($0<\epsilon \ll 1)$,  we have
\[
\mathrm{Im}\l(\phi_\pm(v\mp \ri \epsilon)\r) = \pm \frac{2\pi}{3}+ \epsilon  \mathrm{Re}\l(\phi_\pm(v)\r) + \cO(\epsilon^2)\,,
\]
where we have used the Cauchy--Riemann equations for $\phi_\pm(u)$. As $\mathrm{Re}\l(\phi_\pm(v)\r)$ is odd and has exactly two critical points, this implies that $x(v\pm \frac{\tau-2 \epsilon}{2}) = \exp \phi_\pm(v\mp \ri \epsilon)$ maps $\cA=(-1/2, 1/2]$ to a simple loop around $\mathfrak{I}_\pm$, which is readily checked to be oriented clockwise since the tangent vectors 
\[
x_* \de_{\mathrm{Re}(u)} =
\de_v \exp \phi_-(v+ \ri \epsilon)|_{v=-1/2}\,, \quad
x_*\de_{\mathrm{Im}(u)} =\de_v \exp \chi_+(v)|_{v=-1/2}\,,
\]
must form a right-handed basis of $T_{x=\varphi^{-1}} \bbC \simeq \bbC$ by conformality of $x(u)$. 
\end{proof}

The identification $t = t(\mathsf{q})$, $x_\pm = \re^{\pm \mathfrak{x}}$ will henceforth be assumed throughout all the discussion to follow.  

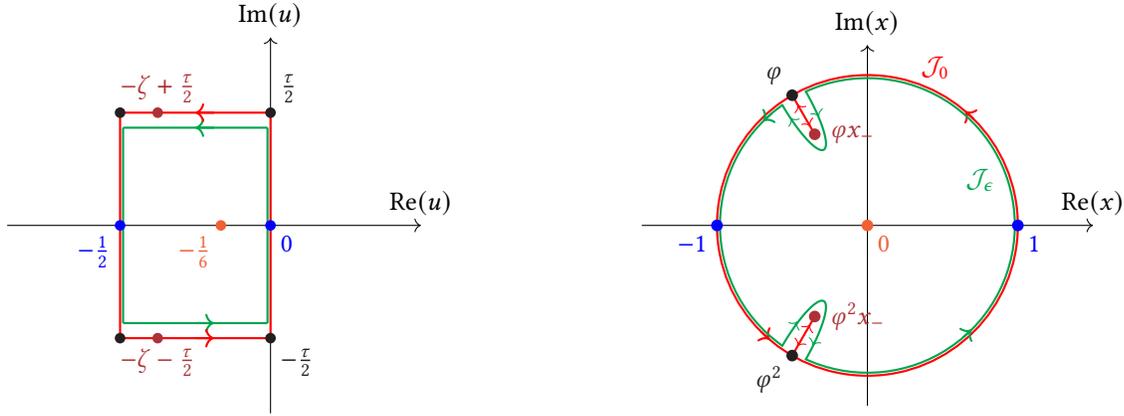
\begin{figure}

\begin{minipage}{.49\textwidth}

\begin{tikzpicture}[scale=1]
  \draw[->] (-3.5, 0) -- (2, 0) node[above] {\small $\mathrm{Re}(u)$};
  \draw[->] (0, -2.5) -- (0, 2.5) node[above] {\small $\mathrm{Im}(u)$};

  \draw[thick, red] (0, 1.5) -- (-2, 1.5);
  \draw[thick, red] (-2, 1.5) -- (-2, -1.5);
  \draw[thick, red] (-2, -1.5) -- (0, -1.5);
  
  \draw[thick, Green] (-1.96, -1.3) -- (-0.04, -1.3);

  \draw[thick, Green] (-0.04, -1.3) -- (-0.04, 1.3);

  \draw[thick, Green] (-0.04, 1.3) -- (-1.96, 1.3);

  \draw[thick, Green] (-1.96, 1.3) -- (-1.96, -1.3);

\draw[->, Green, thick] (-1, -1.3) -- (-.75, -1.3);

\draw[->, Green, thick] (-.75, 1.3) -- (-1, 1.3);

\draw[->, red, thick] (-1, -1.5) -- (-0.75, -1.5);

\draw[->, red, thick] (-0.75, 1.5) -- (-1, 1.5);

  \draw[thick, red] (0, -1.5) -- (0, 1.5);

  \fill[blue] (0, 0) circle (2pt) node[below right] {\small $0$};

  \fill[blue] (-2,0) circle (2pt) node[below left] {\small $-\frac{1}{2}$};

  \fill[RedOrange] (-.66,0) circle (2pt) node[below left] {\small $-\frac{1}{6}$};


    \fill[Maroon] (-1.5,-1.5) circle (2pt) node[below] {\small $-\mathfrak{u}-\frac{\tau}{2}$};

    \fill[Maroon] (-1.5,1.5) circle (2pt) node[above] {\small $-\mathfrak{u}+\frac{\tau}{2}$};

    \fill[Black] (-2,-1.5) circle (2pt); 

    \fill[Black] (-2,1.5) circle (2pt);
    
    \fill[Black] (0,-1.5) circle (2pt) node[below right] {\small $-\frac{\tau}{2}$};

    \fill[Black] (0,1.5) circle (2pt) node[above right] {\small $\frac{\tau}{2}$};

\end{tikzpicture}

\end{minipage}
\begin{minipage}{.5\textwidth}
\begin{tikzpicture}[scale=2]
  \draw[->] (-1.5, 0) -- (1.5, 0) node[above] {\small $\mathrm{Re}(x)$};
  \draw[->] (0, -1.2) -- (0, 1.2) node[above] {\small $\mathrm{Im}(x)$};

  \draw[thick, red] (0,0) circle(1);
  \draw[->, red, thick] (0.707, 0.707) -- (0.649, 0.760);

  \draw[->, red, thick] (-0.707, -0.707) -- (-0.649, -0.760);
  
  \foreach \angle/\label in {120/{$\omega$}, 240/{$\varphi^2$}} {

    \draw[thick, red] 
      ({cos(\angle)}, {sin(\angle)}) -- 
      ({0.7*cos(\angle)}, {0.7*sin(\angle)});
          \draw[red, <->] 
      ({0.93*cos(\angle)}, {0.93*sin(\angle)}) -- 
      ({0.76*cos(\angle)}, {0.76*sin(\angle)});

  }

  \coordinate (A) at ({0.98*cos(240+5)}, {0.98*sin(240+5)});
  \coordinate (B) at ({0.98*cos(120-5)}, {0.98*sin(120-5)});
  \coordinate (C) at ({(1 - 0.25)*cos(120)}, {(1 - 0.25)*sin(120)});
  \coordinate (D) at ({0.98*cos(120+5)}, {0.98*sin(120+5)});

  \coordinate (E) at ({0.98*cos(120+5)}, {0.98*sin(120+5)});
  \coordinate (F) at ({0.98*cos(240-5)}, {0.98*sin(240-5)});
  \coordinate (G) at ({(1 - 0.25)*cos(240)}, {(1 - 0.25)*sin(240)});
  \coordinate (H) at ({0.98}, {0});

  \draw[Green, thick]
    (F) arc[start angle=235, end angle=125, radius=0.98];

  \draw[Green, thick]
    (B) arc[start angle=115, end angle=0, radius=0.98];

  \draw[Green, thick]
    (H) arc[start angle=0, end angle=-115, radius=0.98];
    
  \def\xa{-0.49}  
  \def\ya{0.85}  
  \def\a{0.4}   
  \def\b{0.09}   
  \def\angle{120} 

  \draw[Green, thick, domain=90:270, samples=100, variable=\t]
    plot ({\xa + \a*cos(\t)*cos(\angle) - \b*sin(\t)*sin(\angle)},
          {\ya + \a*cos(\t)*sin(\angle) + \b*sin(\t)*cos(\angle)});

   \draw[Green, ->] ({0.82* cos(\angle-6)}, {0.82*sin(\angle-6)}) --  ({0.79* cos(\angle-6)}, {0.79*sin(\angle-6)});

   \draw[Green, ->] ({0.85* cos(\angle+5.5)}, {0.85*sin(\angle+5.5)}) --  ({0.88* cos(\angle+5.5)}, {0.88*sin(\angle+5.5)});
  \def\xa{-0.49}  
  \def\ya{-0.85}  
  \def\a{0.4}   
  \def\b{0.09}   
  \def\angle{240} 

  \draw[Green, thick, domain=90:270, samples=100, variable=\t]
    plot ({\xa + \a*cos(\t)*cos(\angle) - \b*sin(\t)*sin(\angle)},
          {\ya + \a*cos(\t)*sin(\angle) + \b*sin(\t)*cos(\angle)});

   \draw[Green, ->] ({0.82* cos(\angle-6)}, {0.82*sin(\angle-6)}) --  ({0.79* cos(\angle-6)}, {0.79*sin(\angle-6)});

   \draw[Green, ->] ({0.85* cos(\angle+5.5)}, {0.85*sin(\angle+5.5)}) --  ({0.88* cos(\angle+5.5)}, {0.88*sin(\angle+5.5)});
   
  \draw[->, Green, thick] (0.649*0.98, -0.760*0.98) -- (0.707*0.98, -0.707*0.98);

    \draw[->, Green, thick] (-0.649*0.98, 0.760*0.98) -- (-0.707*0.98, 0.707*0.98);

  \fill[RedOrange] (0,0) circle (1.1pt) node[below right] {\small $0$};

  \fill[Black] (-0.5,0.866) circle (1.1pt) node[above left] {\small $\varphi$};

  \fill[Black] (-0.5,-0.866) circle (1.1pt) node[below left] {\small $\varphi^2$};

  \fill[Maroon] (-0.5*0.7,0.866*0.7) circle (1.1pt) node[ right] {\small $~\varphi x_-$}; 

  \fill[Maroon] (-0.5*0.7,-0.866*0.7) circle (1.1pt) node[ right] {\small $~\varphi^2 x_-$}; 

  \fill[blue] (-1,0) circle (1.1pt) node[below left] {\small $-1$};

  \fill[blue] (1,0) circle (1.1pt) node[below right] {\small $1$};


\end{tikzpicture}
\end{minipage}
\caption{The boundary rectangles of $\Delta_-$ (left, red) and $\Delta^{[\epsilon]}_-$ with $\epsilon>0$ (left, green), and their images 
 under $x(u)$ (right)
.}
\label{fig:fund_dom_hol}
\end{figure}

\begin{lem}
The restriction of $x(u)$ to $\overline{\Delta}_-$ is holomorphic on $\Delta_-$, continuous on $\de \overline{\Delta}_-$, and surjective 
onto the closed unit disc $\overline{\mathbb{D}}$. 
\label{lem:xusurj}
\end{lem}

\begin{proof}
Holomorphicity on $\Delta_-$ and continuity along the boundary is obvious, since $x(u)$ has only a simple pole in $\Delta$ at $u=1/6$, and is therefore regular for $\mathrm{Re}(u)\leq 0$. 
By \cref{lem:phichi,prop:CtoT}, we know that $x(u)$ maps $\de\overline{\Delta}_-$ surjectively onto the union of the unit circle with the two segments 
\[
\mathfrak{I}_\pm^- \coloneqq \mathfrak{I}_\pm \cap \mathbb{D}\,,
\]
corresponding to the portion of the cuts $\mathfrak{I}_\pm$ contained in the unit disc. 
In particular, we have $|x(u)| \leq 1$ on $\de\overline{\Delta}_-$, implying that $|x(u)|<1$ across $\Delta_-$ by the Maximum Modulus Theorem. 
Surjectivity onto the interior of the unit disc would follow from proving that $x(\Delta_-)\setminus (\mathfrak{I}^-_+  \cup \mathfrak{I}^-_-)$ is equal to the open unit disc with the two segments 
$\mathfrak{I}_\pm^-$
removed. To this aim note that, by holomorphicity and the Open Mapping Theorem, $x(\Delta_-)\setminus (\mathfrak{I}^-_+  \cup \mathfrak{I}^-_-)$ is open in $\mathbb{D} \setminus (\mathfrak{I}^-_+  \cup \mathfrak{I}^-_-)$. By continuity, $x(\overline{\Delta}_-)$ is compact, and \[x(\overline{\Delta}_-)\setminus (\de\overline{\mathbb{D}} \cup \mathfrak{I}^-_+  \cup \mathfrak{I}^-_-) = x(\Delta_-)\setminus (\mathfrak{I}^-_+  \cup \mathfrak{I}^-_-)\] is also closed in  $\mathbb{D} \setminus (\mathfrak{I}^-_+  \cup \mathfrak{I}^-_-)$. As $\mathbb{D} \setminus (\mathfrak{I}^-_+  \cup \mathfrak{I}^-_-)$ is connected, and $x(\Delta_-) \neq \emptyset$, this implies $x(\Delta_-)= \mathbb{D} \setminus (\mathfrak{I}^-_+  \cup \mathfrak{I}^-_-)$, and $x(\overline{\Delta}_-)=\mathbb{D}$.

\end{proof}
\begin{cor}
$x:\Delta \to \widehat{\bbC}$ is holomorphic and surjective onto the Riemann sphere.
\end{cor}
\begin{proof}
This follows from $x(u)=x(-u)^{-1}$, and \cref{lem:xusurj}.
\end{proof}




\subsubsection{Solving the Riemann Hilbert problem}
To summarise the discussion so far, we have proved that for $0< \mathsf{q} \ll 1$ the map $x(u)$ in \eqref{eq:xu} is a meromorphic function on $\bbC$, quasi-periodic along $M$, periodic along $\bbZ \subset \bbR$, and whose restriction to the fundamental domain of $M$ is surjective onto $\widehat{\bbC}$. Furthermore, under the identification $t = t(\mathsf{q})$ of \cref{prop:CtoT,rmk:tq}, the images of the fundamental edges of $\Delta$ map to the base circle and axis of the cylinder $\bbR \times S^1$, obtained by excising the cuts $\mathfrak{I}_\pm \subset \bbC$ and identifying them antipodally as in \cref{fig:glueing}. \\ The main upshot of all the preparatory work in the previous Section was then to provide us with a holomorphic version of \eqref{eq:diagT} with respect to the respective complex structures. This puts us now in a position to solve the RHP 
for $\cW(x)$ (and therefore $\RR_{0,1}(x)$) using Riemann--Roch.
\begin{prop}
Under the change-of-variables $t = t(\mathsf{q})$ in \cref{prop:CtoT}, 
the symmetrised planar resolvent $\RR_{0,1}(x)$ reads
\beq
\RR_{0,1}(x(u)) = 
\frac{\ri}{\sqrt{3} }  
 \log\frac{(x(u)+1) \vartheta_1(u-1/6)}{\vartheta_1(u-1/2)}\,.
\label{eq:r01}
\eeq
\end{prop}
\begin{proof}
By 
{\bf RHP(i)--(iv)} and \cref{prop:CtoT}, the pull-back of $\cW$ to $T$ under the diagram \eqref{eq:diagT}  is a single-valued and meromorphic function on $T$ 
with simple poles at $u=\frac{1}{2}$ (i.e. $x=-1$) and $u= \pm \mathfrak{u} - \frac{\tau}{2}$ (i.e. $x=\varphi^{\pm 1} x_\pm)$,
\[
\cW \in H^0\l(T, \cO(\mathsf{p}_{1/2}+\mathsf{p}_{-\mathfrak{u}-\tau/2}+\mathsf{p}_{\mathfrak{u}-\tau/2})\r)\,.
\]
It has, moreover, a simple zero at $u=-\frac{1}{6}$ (i.e. $x=0$), by {\bf RHP(v)}. 
Likewise, by \eqref{eq:xu},
\[
\de_u \log x \in H^0\l(T, \cO(\mathsf{p}_{1/6}+\mathsf{p}_{-1/6})\r)
\]
with $x'(u)/x(u)$ an elliptic function having simple poles with residue $\mp 1$ at $\pm 1/6$, and simple zeroes at the ramification points of the cover $u=\pm\mathfrak{u}-\tau/2$, where $x'(u)$ vanishes. 
Then, the product
\[
r \coloneqq \cW \otimes \de_u \log x \in H^0\l(T, \cO(\mathsf{p}_{1/2}+\mathsf{p}_{1/6})\r)\,,
\]
is meromorphic with only poles at $u=1/2$ and $u=1/6$.
%
%
In $\bbC^{n+1}$ with the standard scalar product and orthonormal basis $\{v_0, \dots, v_n\}$, the space of meromorphic functions on $T$ with simple poles at $\mathsf{p}_{u_i}$, by Riemann--Roch and Kodaira vanishing,  is the hyperplane
\beq
H^0\l(T, \cO(\mathsf{p}_{u_1}+\mathsf{p}_{u_n})\r)
 = \l(\sum_{i=1}^n v_i\r)^\perp 
\label{eq:h0vi}
\eeq
upon identifying 
\[
v_0 = 1\,, \qquad v_i = \frac{\vartheta_1'(u-u_i)}{\vartheta_1(u-u_i)}\,,
\]
the constraint \eqref{eq:h0vi} arising from the residue theorem (equivalently, from imposing that a linear combination of $v_i$ be periodic under $u \to u+\tau$).
In our case, this implies that
\[
r(u) = \gamma \l(\frac{\vartheta_1'(u-1/2)}{\vartheta_1(u-1/2)}-\frac{\vartheta_1'(u-1/6)}{\vartheta_1(u-1/6)}
\r)+\delta
\]
for yet-to-be-fixed constants $\gamma$, $\delta \in \bbC$. Imposing that, from {\bf RHP(iii)}, the meromorphic 1-form $\cW(x) \rd \log x$ has residue $1/(\sqrt{3}\ri)$ at $x=-1$ 
sets
\[
\frac{1}{\sqrt{3} \ri}=\Res_{x=-1} \cW(x) \rd \log x  =\Res_{u=1/2} r(u) \rd u = \gamma \,.
\]
Further imposing that, from {\bf RHP(v)}, $r(u)$ has a zero at $u=-1/6$ yields
\[
\delta=\frac{1}{\sqrt{3} \ri}
\l(\frac{\vartheta_1'(-1/3)}{\vartheta_1(-1/3)}-\frac{\vartheta_1'(-2/3)}{\vartheta_1(-2/3)}
\r)
=0\,,
\]
where we have used \eqref{eq:thetaper}. All in all, we find
\beq
\RR_{0,1}(x(u)) = \int^{u} r(u') \rd u' -\frac{1}{\sqrt{3}  \ri} \int^{x(u)} \frac{\rd x'}{x'+1} =
\frac{\ri}{\sqrt{3} }  
 \log\frac{(x(u)+1) \vartheta_1(u-1/6)}{\vartheta_1(u-1/2)} + \psi
\label{eq:r01sol}
\eeq
for some constant $\psi \in \bbC$. Imposing that $\RR_{0,1}(x) = \cO(x)$ near $x=0$ ($u=-1/6$), and using again that $\vartheta_1(-2/3)=\vartheta_1(-1/3)$, sets $\psi=0$.
\end{proof}
\begin{rmk}
With \eqref{eq:r01} at hand, we can write down the coefficients of the change-of-variables $t = t(\mathsf{q})$ by integrating term-by-term along $\cA$ the expression of $\RR_{0,1}(x(u))\rd u$, as a 1-form valued uniformly convergent $\mathsf{q}$-series, using \eqref{eq:thetaper} and imposing {\bf RHP(vii)}. In particular, this entails that the map $t(\mathsf{q})$ is 
real-analytic near $t=\mathsf{q}=0$. We find
\[
t(\mathsf{q}) = 3 \mathsf{q}+\frac{9 \mathsf{q}^2}{2}+9 \mathsf{q}^3+\frac{39 \mathsf{q}^4}{4}+\frac{72 \mathsf{q}^5}{5}+\cO\left(\mathsf{q}^6\right)\,,
\]
so that, imposing that $t^{1/2} \kappa(t) = \mathsf{q}^{1/2} \mu(q)$ and using \eqref{eq:muq}, the small positive 't~Hooft coupling expansion of the branch-points $x_\pm(t)$ reads
\[
\log x_\pm(t) = \pm \sqrt{t} \l[
2 +\frac{t}{2}-\frac{3 t^{2}}{5}+\frac{115 t^{3}}{3024}-\frac{341 t^{4}}{864}+\cO\left(t^{5}\right)
\r]\,.
\]
\label{rmk:tq}
\end{rmk}

\subsection{Riemann--Hilbert problem for the planar two-point function}
A similar strategy applies to the symmetrised planar two-point function, which can be characterised as the solution of a particular Riemann--Hilbert problem in each argument. From \eqref{eq:lambda}, we have
\begin{align}
& \frac{x_2}{\pi\ri}\oint_{\mf{I}_0} x_1\de_{x_1}\Lambda(x_1, \zeta)  W^{(\Phi,\Lambda)}_{0,2}(x_2,\zeta)\rd\zeta = \nn \\ 
& -\frac{x_2}{2\pi\ri}\oint_{\mf{I}_0} \bigg[\frac{x_1 \varphi}{x_1 \varphi-\zeta}+\frac{x_1 \varphi^2}{x_1 \varphi^2-\zeta}\bigg]W^{(\Phi,\Lambda)}_{0,2}(x_2,\zeta)\rd\zeta\,, \nn \\
& =
-x_1 x_2 \varphi W^{(\Phi,\Lambda)}_{0,2}(\varphi x_1, x_2)
-x_1 x_2 \varphi^2 W^{(\Phi,\Lambda)}_{0,2}(\varphi^2 x_1, x_2)\,,
\label{eq:W02conv}
\end{align}
and plugging \eqref{eq:W02conv} into
\eqref{eq:cutom02} we find that,
for all $x_1 \in \mathfrak{I}_0$,
\begin{align}
&    \RR_{0,2}(\varphi^{2} x_1 \re^{+\ri 0}, x_2)-    \RR_{0,2}(\varphi x_1 \re^{-\ri 0}, x_2) = \nn \\
= &
x_1 x_2 \bigg[
\varphi    \, W^{(\Phi,\Lambda)}_{0,2}( x_1 \re^{+\ri 0}, \varphi x_2) -  \varphi^2  \, W^{(\Phi,\Lambda)}_{0,2}(x_1 \re^{+\ri 0}, \varphi^2 x_2) \nn \\
& -  \varphi^2  \, W^{(\Phi,\Lambda)}_{0,2}(\varphi x_1, \varphi x_2)
+  \, W^{(\Phi,\Lambda)}_{0,2}(\varphi x_1, \varphi^2 x_2) 
-  
    \,
 \, W^{(\Phi,\Lambda)}_{0,2}(\varphi^2 x_1, \varphi x_2)  \nn \\ & +   \, \varphi 
 W^{(\Phi,\Lambda)}_{0,2}(\varphi^2 x_1, \varphi^2 x_2)+ \varphi  \, W^{(\Phi,\Lambda)}_{0,2}(x_1 \re^{-\ri 0}, \varphi x_2)
-\varphi^2  \, W^{(\Phi,\Lambda)}_{0,2}(x_1 \re^{-\ri 0}, \varphi^2 x_2)\bigg]
\,,
\nn \\
= &
\frac{\varphi x_1 x_2 }{(x_1-\varphi x_2)^2}-
\frac{\varphi^2 x_1 x_2 }{( x_1-\varphi^2 x_2)^2}=\frac{\varphi^2 x_1 x_2 }{(\varphi^2 x_1- x_2)^2}-
\frac{\varphi x_1 x_2 }{( \varphi x_1-\varphi^2 x_2)^2}\,.
\label{eq:speqR02-1}
\end{align}
Symmetrically, for $x_2 \in \mathfrak{I}_0$,
\begin{align}
 \RR_{0,2}(x_1, \varphi^{2} x_2 \re^{+\ri 0})-    \RR_{0,2}(x_1, \varphi x_2 \re^{-\ri 0}) = &
\frac{\varphi^2 x_1 x_2 }{(\varphi^2 x_2- x_1)^2}-
\frac{\varphi x_1 x_2 }{( \varphi x_2-\varphi^2 x_1)^2}
\nn \\
=&
\frac{\varphi x_1 x_2 }{(\varphi x_1- x_2)^2}-
\frac{\varphi^2 x_1 x_2 }{(\varphi^2 x_1-x_2)^2}
\,.
\label{eq:speqR02-2}
\end{align}
Eqs.~\eqref{eq:speqR02-1}--\eqref{eq:speqR02-2} are the analogue of the saddle-point equation \eqref{eq:RHR01der} for the $\log x$-derivative of the symmetrised planar resolvent. As we did in \eqref{eq:Wdef}--\eqref{eq:Wspeq}, we homogenise \eqref{eq:speqR02-1}--\eqref{eq:speqR02-2} by defining
\beq
\cY(x_1, x_2) \coloneqq \RR_{0,2}(x_1, x_2) + \frac{x_1 x_2}{(x_1-x_2)^2}\,,
\label{eq:Y}
\eeq 
in terms of which they read
\begin{align}
\cY(\varphi^{\pm} x_1 \re^{\pm\ri 0}, x_2)-    \cY(\varphi^{\pm} x_1 \re^{\mp\ri 0}, x_2)=0\,, \qquad x_1 \in \mathfrak{I}_0\,, \nn \\
\cY(x_1, \varphi^{\pm} x_2 \re^{\pm\ri 0})-    \cY(x_1,\varphi^{\pm} x_2 \re^{\mp\ri 0})=0\,, \qquad x_2 \in \mathfrak{I}_0\,.
\label{eq:RHPY}
\end{align}
By the arguments used in the analysis of $\cW$ in \cref{sec:RHR01}, repeated \emph{verbatim} here, $\cY$ pulls back via $x_1(u_1)$, $x_2(u_2)$ to a meromorphic function on $T\times T$, symmetric under $u_1 \leftrightarrow u_2$, and therefore to a meromorphic function on the symmetric square $T^{(2)}$ of the elliptic curve $T$.
%
More in detail, by \eqref{eq:Y}, $\cY$ is regular away from the diagonal $x_1=x_2$, where it has a double pole with leading coefficient 1,
\[
\cY \in H^0\l(T^{(2)}, \cO_{T^{(2)}}(2\Delta)\r)\,.
\]
A classical result in the theory of elliptic ruled surfaces is that $|2\Delta|$ is a base-point-free pencil of elliptic curves (see e.g. \cite[Prop.~3.2]{MR1423277}), and thus  \[ h^0(T^{(2)},\cO_{T^{(2)}}(2\Delta))=2\,.\] 
From \eqref{eq:wp}, explicit generators for $\Gamma(\cO_{T^{(2)}}(2\Delta))$ are given by $\l\{1, \wp(u_1-u_2)\r\}$. Further recalling that
\[
\cY = \frac{x_1 x_2}{(x_1-x_2)^2}+\cO((x_1-x_2)^0)\,,\]
we find
\[
\cY(x(u_1), x(u_2)) \frac{x'(u_1)x'(u_2)}{x(u_1) x(u_2)} = \frac{1}{(u_1-u_2)^2} + \cO((u_1-u_2)^0)\,, 
\]
hence
\beq
\cY(x_1, x_2) \frac{x'(u_1)x'(u_2)}{x(u_1) x(u_2)} = \wp(u_1-u_2) + \xi
\label{eq:2pwp}
\eeq
for $\xi \in \bbC$. From 
\[
\oint_{\alpha} \RR_{0,2}(x_1,x_2) \frac{\rd x_1}{x_1 x_2}=\oint_{\alpha} \cY(x_1,x_2) \frac{\rd x_1}{x_1 x_2} =0
\]
we find
\beq
\xi = -\int_0^1 \wp(u_1-u_2) \rd u_1 = \frac{\pi^2}{3} E_2(\mathsf{q})\,,
\label{eq:2ptorelli}
\eeq
where $E_2(\mathsf{q})$ denotes as usual the second Eisenstein series:
\[
E_2(\mathsf{q}) \coloneqq 1-24 \sum_{n\geq 1} \frac{n \mathsf{q}^{n}}{1-\mathsf{q}^{n}}\,.
\]
We have proved the following:
\begin{prop}
Under the change-of-variables $t = t(\mathsf{q})$ in \cref{prop:CtoT,rmk:tq}, 
the symmetrised planar 2-point function $\cS_\chi\omega^{(\Phi,\Lambda)}_{0,2}$  is
\beq
\cS_\chi\omega^{(\Phi,\Lambda)}_{0,2}  
= \l[\wp(u_1-u_2)+
\frac{\pi^2}{3} E_2(\mathsf{q}) 
\r] \rd u_1 \rd u_2
\,.
\label{eq:r02}
\eeq
\end{prop}

%
%
%

\subsection{The Eynard--Orantin recursion} Our final result in this section establishes the Eynard--Orantin recursion for the symmetrised correlators $\cS_\chi\omega_{g,k}^{(\Phi, \Lambda)}$ in the form given in \cref{prop:EOsymm}.

\begin{prop}
For all $2g-2+k>0$, and under the change-of-variables $t = t(\mathsf{q})$ in \cref{prop:CtoT,rmk:tq}, $\cS_\chi\omega_{g,k}^{(\Phi, \Lambda)}$ is a symmetric meromorphic multi-differential on the elliptic curve $T$. Moreover, $\{\cS_\chi\omega_{g,k}^{(\Phi, \Lambda)}\}_{g \geq 0, k \geq 1}$ 
satisfy quadratic loop equations. 
\label{prop:EOMM}
\end{prop}

\begin{proof}
As we did for $(g,k)=(0,1)$ and $(0,2)$, we turn around the contour in the linear loop equation \eqref{eq:LLEun} 
and pick up the residues at $\zeta= \varphi^{\pm 1} z$. The same arguments used for \eqref{eq:Wspeq} and \eqref{eq:RHPY} entail that, for $2g-2+k>0$, $i=1,\dots, k$ and $x_i \in \mathfrak{I}_0$,
\beq
\RR_{g,k}(x_1, \dots, \varphi^2 x_i \re^{\mp\ri 0}, \dots, x_k)=
\RR_{g,k}(x_1, \dots, \varphi x_i \re^{\pm\ri 0}, \dots, x_k)\,.
\label{eq:RgkRHP}
\eeq
Thus,
\[
\cS_\chi \omega_{g,k}^{(\Phi,\Lambda)} = \RR_{g,k}(x(u_1), \dots, x(u_k)) \otimes_{i=1}^k  \frac{x'(u_i) \rd u_i}{x(u_i)}
\]
is a single-valued symmetric meromorphic multi-differential on $T$, with only degree $6g-4+2k$ poles at the ramification points $u=\pm\mathfrak{u} \pm \tau/2$.  To prove the second assertion, for $I=\llbracket 3, k+2 \rrbracket$, we define
\[
\frac{\Theta_{g,k}(x_1, x_2, x_I)}{x_1 x_2} \coloneqq  
W^{(\Phi,\Lambda)}_{g-1,k+2}(x_1,x_2, x_I)+\sum_{J \subseteq I, 0\leq h \leq g} W^{(\Phi,\Lambda)}_{h,|J|+1}(x_1, x_J) W^{(\Phi,\Lambda)}_{g-h,k-|J|+1}(x_2, x_{I \setminus J})\,,
\]
\begin{align*}
\Theta^\chi_{g,k} (x_1, x_2, x_I) \coloneqq \Theta_{g,k}(\varphi x_1, \varphi x_2,x_I)- \Theta_{g,k}(\varphi^2 x_1, \varphi x_2,x_I)\\- \Theta_{g,k}(\varphi x_1, \varphi^2 x_2,x_I) 
+ \Theta_{g,k}(\varphi^2 x_1, \varphi^2 x_2,x_I)\,.
\label{eq:Sx1x2f}
\end{align*}
We want to evaluate $\Theta^\chi_{g,k}(x_1, x_2, x_I)$ at
 $(x_1,x_2) = (\varphi^{\pm 1} x\re^{+\ri 0}, \varphi^{\pm 1} x \re^{-\ri 0})$, with $x\in \mathfrak{I}_0$. 
 Using \eqref{eq:RHR01}, \eqref{eq:speqR02-1} and \eqref{eq:RgkRHP}, we can trade this for an evaluation at $(x_{1},x_2)=( \varphi^{\pm 1} x\re^{+\ri 0},\varphi^{\mp 1} x \re^{+\ri 0})$ instead,
 modulo 
 additional terms arising from the r.h.s. of \eqref{eq:RHR01} (resp. \eqref{eq:speqR02-1}) when $g=h$ and 
 $|J|=k$ (resp. $|J|=k-1$).
More precisely, for $x \in \mathfrak{I}_0$ we get
\begin{align*}
\Theta^\chi_{g,k}(\varphi^{\pm 1} x \re^{+\ri 0} , \varphi^{\pm 1} x\re^{-\ri 0}, x_I)
& =  - \Theta_{g,k}(x\re^{+\ri 0},  x\re^{+\ri 0}, x_I)
\nn \\ & +\Theta_{g,k}(\varphi^{\mp 1} x, x\re^{+\ri 0}, x_I)+ \Theta_{g,k}(x\re^{+\ri 0}, \varphi^{\pm 1} x, x_I) 
\nn \\ & + {x^2 \Phi'(x)} W^{(\Phi,\Lambda)}_{g,k+1}(x\re^{+\ri 0}, x_I) - \sum_{i \in I} \frac{x^2 W^{(\Phi,\Lambda)}_{g,k}(x\re^{+\ri 0}, x_{I \setminus \{i\}})}{(x-x_i)^2} \nn \\ 
& - \Theta_{g,k}(\varphi^{\mp 1} x, \varphi^{\pm 1} x, x_I)
\,.
\end{align*}
Up to an overall factor of $-x^{-2} (\rd x)^2 \prod_{i=3}^{k+2} \rd x_i$, the first three lines of the r.h.s. coincide with the first three lines of the r.h.s. of the Schwinger--Dyson equation\footnote{As discussed below \eqref{eq:SD}, the case $(g,k)=(1,1)$ has a mild anomaly which in our case amounts to replacing every appearance of $\RR_{0,2}(x_1, x_2)$ in $\Theta_{0,2}^\chi(x_1,x_2)$ by $\cY(x_1,x_2)$. Thence, the proof for general $(g,k) \neq (1,1)$ carries through to this special case as well.}  \eqref{eq:SD}. We deduce that
\[
\Theta^\chi_{g,k}(\varphi^{\pm 1} x \re^{+\ri 0},  \varphi^{\pm 1}x\re^{-\ri 0}, x_I) = - \Theta_{g,k}(\varphi^{\mp 1} x, \varphi^{\pm 1} x, x_I) + \Pi_{g,k}(x, x_I)\,,
\]
for $\Pi_{g,k}(x,x_I)$ a holomorphic function of $x$ near $x=x_\pm$. By \cref{prop:EO}, the first summand in the r.h.s.
is also regular at $x=x_\pm$. Therefore, multiplying both sides by $x^{-2} (\rd x)^2 \prod_{i=3}^{k+2} \rd x_i$, and applying $\cS_\chi$ to the spectator variables $x_I$, we deduce that
\[
\cS_\chi \omega^{(\Phi,\Lambda)}_{g-1,k+2}(p, \iota(p), p_I) + \sum_{J \subseteq I, 0\leq h \leq g} \cS_\chi \omega^{(\Phi,\Lambda)}_{h,|J|+1}(p, p_J) \cS_\chi \omega^{(\Phi,\Lambda)}_{g-h,k-|J|+1}(\iota(p), p_{I \setminus J})
\]
has quadratic zeroes at $x(p) \in \{\varphi^{n} x_\pm\}_{n \in \{1,2\}}$, concluding the proof.
\end{proof}

\begin{cor}
For $2g-2+k>0$, the $\chi$-equivariant correlators $\cS_\chi \omega_{g,k}^{(\Phi,\Lambda)}$ satisfy the Eynard--Orantin recursion \eqref{eq:EOsymm}.
\label{cor:EOMM}
\end{cor}

\section{The Conifold Gap}
\label{sec:lp2}
\subsection{Gromov--Witten invariants of local $\bbP^2$}

We write $K_{\bbP^2}$ for the total space $\mathrm{Tot}_{\bbP^2}(\omega_{\bbP^2})$ of the canonical bundle on the projective plane, 
and denote $H$ the pull-back to $K_{\bbP^2}$ of the class of a line in the cohomology of $\bbP^2$. For $d>0$, 
the moduli stack $\overline{M}_{g}(K_{\bbP^2}, d\,H)$ of genus-$g$, degree $d H$ stable maps to $K_{\bbP^2}$ is proper
  and of virtual dimension zero, with
  virtual fundamental class:
\[
[\overline{M}_{g}(K_{\bbP^2}, d\,H)]^{\rm vir}  \coloneqq [\overline{M}_{g}(\bbP^2, d\,H)]^{\rm vir} \cap c_{\rm top}(-R^\bullet \pi_* \mathrm{ev}^* \cO_{\bbP^2}(-3))
\,.
\]
The large radius genus-$g$ GW potential of $K_{\bbP^2}$ is the generating function of the degrees of the virtual fundamental class for each fixed $g\geq 0$:
\[
\mathrm{GW}_g^{\rm LR}(Q) \coloneqq \sum_{d >0} Q^d N_{g, d H}^{K_{\bbP^2}} \coloneqq  \int_{[\overline{M}_g(K_{\bbP^2}, d H)]^{\rm vir}} 1\,.
\]
%

\subsection{Higher genus local mirror symmetry}


\label{sec:HV}

Recall that the Hori--Iqbal--Vafa mirror family of  $K_{\bbP^2}$ is described by the family of elliptic curves over the weighted projective line $\bbP(3,1)$ \cite{Hori:2000ck, Coates:2018hms},
\[
\cZ = \l\{\l([z_0:z_1:z_2:z_3], y\r) \in \bbP^3 \times \bbP(3,1)\, \Big|\, \sum_{i=0}^3 z_i=0\,,\,   z_1 z_2 z_3=y z_0^3 \r\} \longrightarrow \bbP(3,1)\,.
\]

The family has singular fibres at the two values $y=-1/27$ (the {\it conifold point}) and $y=0$ (the {\it large radius point}), away from which the fibres are smooth elliptic curves, given by the normalisation of the closure in $\bbP^2$ of
\beq
\Big\{(z_1,z_2) \in (\bbC^\times)^2 \,\Big|\, z_1+z_2+\frac{y}{z_1 z_2}+1=0 \Big\}\,.
\label{eq:HV}
\eeq
For such a regular fibre $\cZ_{y}$ over $y \notin \{0, -1/27\}$, write
\[
\Pi_\gamma(y) = \frac{1}{2\pi\ri}\int_\gamma \log z_1 \rd \log z_2
\]
for the period of the (multi-valued) 1-form $\log z_1 \rd \log z_2$ over $\gamma \in H_1(\cZ_y, \bbC)$.
Around the large radius point $y=0$, 
there are unique homology classes $\alpha_{\rm LR}$, $\beta_{\rm LR}$ such that \cite{CKYZ99:locMS,HKR:HAE}
\[
\Pi_{\alpha_{\rm LR}} \sim \frac{1}{2\pi\ri} \log y + \cO(y)\,, \quad \Pi_{\beta_{\rm LR}} \sim \frac{1}{(2\pi\ri)^2}\l[\frac{1}{2}\log^2 y - 
 2\pi\ri \Pi_{\alpha_{\rm LR}} \log (-y)\r]+ \cO(y)\,.
\]
Around the conifold point $y=-1/27$, there are unique homology classes $\alpha_{\rm CF}=-3\beta_{\rm LR }$, $\beta_{\rm CF}=\alpha_{\rm LR}$ such that \cite{HKR:HAE} %
\[
\Pi_{\alpha_{\rm CF}} \sim
\frac{\sqrt{3}\left(1+27 y\right)}{2\pi } + \cO(1+27 y)^2
\,, \qquad \Pi_{\beta_{\rm CF}} \sim
\Pi_{\alpha_{\rm CF}} \frac{\log \left(1+27 y \right)}{2\pi \ri}   + \cO(1)
\,.
\]
%
The complement of the conifold and large radius points is the coarse moduli space $Y_1(3)$ of the modular curve $ [\mathbb{H}/\Gamma_1(3)]$, with
$\mathbb{H} = \{\tau \in \bbC\,|\, \mathrm{Im}(\tau) > 0\}$ the upper half-plane, and 
\[
\Gamma_1(3) \coloneqq \left\{ \begin{pmatrix} a & b \\ c & d \end{pmatrix} \in \mathrm{SL}_2(\mathbb{Z}) \;\middle|\; a,d \equiv 1\,,\, c \equiv 0 \pmod{3} \right\}\,.
\]
Explicitly, we identify $Y_1(3)$ with $\bbP(3,1) \setminus \{y=0,-1/27\}$ using the large radius modular coordinate 
\[
\tau_{\rm LR}(y) \coloneqq \frac{ \de_y \Pi_{\beta^{\rm LR}}}{\de_y \Pi_{\alpha^{\rm LR}}}=\frac{1}{2\pi \ri} \l(\log y-15 y+\frac{333 y^2}{2}+\cO\left(y^3\right)\r)\,.
\]
This is related to the modular coordinate near the conifold point by the Fricke involution \cite{HKR:HAE,Alim:2013eja}
\bea
\tau_{\rm CF} &\coloneqq & \frac{ \de_y \Pi_{\beta^{\rm CF}}}{ \de_y \Pi_{\alpha^{\rm CF}}} = -\frac{1}{ 3\tau_{\rm LR}}
\,.
\label{eq:taucf}
\eea
In terms of the square of the respective elliptic nomes, $q_\bullet \coloneqq \exp(2 \pi\ri \tau_\bullet) $ ($\bullet \in \{\rm LR, CF\}$), the Fricke involution corresponds to the transformation \eqref{eq:qfricke}.
\begin{figure}[t]
\centering
\begin{tikzpicture}[scale=5]

  \draw[->] (-0.6, 0) -- (0.6, 0) node[right] {$\mathrm{Re}\tau_{\rm LR}$};
  \draw[->] (0, 0) -- (0, 1.5) node[above] {$\mathrm{Im}\tau_{\rm LR}$};

  \draw[thick] (-0.5, 0.29) -- (-0.5, 1.5);
  \draw[thick] (0.5, 0.29) -- (0.5, 1.5);

  \draw[thick, domain=0:120] plot ({-1/3+2/3*cos(\x)/2}, {2/3*sin(\x)/2});
  \draw[thick, domain=0:120] plot ({1/3-2/3*cos(\x)/2}, {2/3*sin(\x)/2});

  \filldraw[black] ({-cos(120)}, {sin(120)}) circle (0.01) node[right] {$\varphi$};
  \filldraw[black] (0, 0) circle (0.01); \draw (0, 0) circle (0.01);
  \filldraw[black] (0, 1.45) circle (0.01); 
  \node at (0.07, 1.45) {~~$\ri\infty$};
  \node at (0.03, -0.04) {$0$};

  \node at (-0.5, -0.08) {$-\frac{1}{2}$};
  \node at (0.5, -0.08) {$\frac{1}{2}$};

\end{tikzpicture}
\caption{The fundamental domain for $\Gamma_1(3)$. The cusps at $\tau_{\rm LR}=\ri \infty$ and $\tau_{\rm LR}=0$ ($\tau_{\rm CF}=\ri \infty$) correspond to the large radius $(y=0)$ and conifold point $(y=-1/27)$ respectively. The cubic elliptic point at $\tau_{\rm LR}=\varphi$ corresponds to the orbifold point $y=\infty$.}
\end{figure}
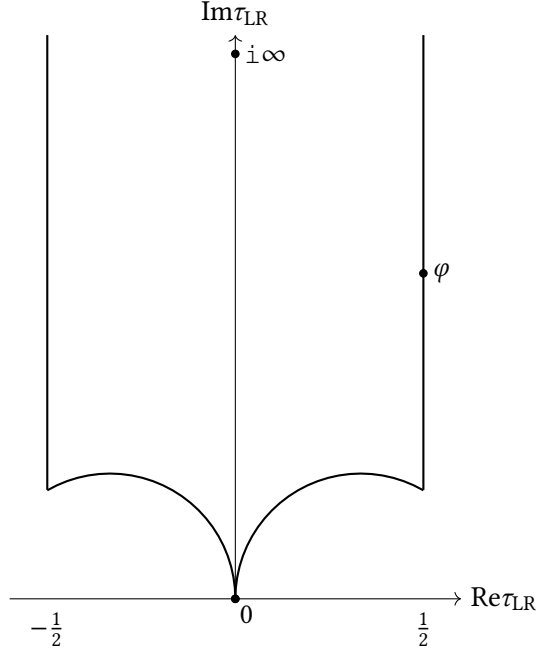
The discriminant $\mathsf{X}(y)$, Jacobian $\mathsf{C}(y)$, and propagator $\mathsf{F}(y)$, defined as
\beq
\mathsf{X}(y) \coloneqq \frac{1}{1+27 y}\,, \qquad \mathsf{C}(y) \coloneqq y \de_y \log \Pi_{\alpha_{\rm LR}}\,, \qquad \mathsf{F}(y) \coloneqq y\de_y \log \mathsf{C}(y)+\frac{1-\mathsf{X}}{3}
\label{eq:XAF}
\eeq
are quasi-modular functions of $\Gamma_1(3)$ of weight and depth equal to  $(0,0)$, $(1,0)$, and $(0,1)$  respectively \cite{Alim:2013eja,BFGW21:HAE}. 
%
In terms of these, the 3-pointed genus zero and 1-pointed genus one GW potential of $K_{\bbP^2}$ read \cite{Aganagic:2006wq,HKR:HAE,Brini:2024gtn}
\[
\l(Q \de_{Q}\r)^3 \mathrm{GW}^{\rm LR}_{0} = \frac{\mathsf{X}}{3 \mathsf{C}^3}\,, \qquad 
Q \de_{Q} \mathrm{GW}^{\rm LR}_{1} = -\frac{\mathsf{F}}{2 \mathsf{C}}-\frac{\mathsf{X}}{12 \mathsf{C}}\,,
\] 
upon substituting $y \longrightarrow y(Q)$, with $y(Q)$ the series inverse of the mirror map $Q(y) \coloneqq \re^{2\pi\ri \Pi_{\alpha_{\rm LR}}} \in \bbQ[\![y]\!]$. In the following, we will also need the \emph{conifold mirror map}
\[
t_{\rm CF} \coloneqq \frac{2 \pi}{\sqrt{3}} \Pi_{\alpha_{\rm CF}} = (1+27 y)+\cO(1+27 y)^2\,,
\]
which is a convergent power series for $|1+27 y|<1$. We will denote by $y(t_{{\rm CF}})$ its Lagrange inverse series in the co-domain of its domain of convergence.

\subsubsection{The Eynard--Orantin recursion and higher genus mirror symmetry}


Recall that a complex Torelli marking on a smooth projective curve $C$ is the choice of a Lagrangian subspace of $H_1(C, \bbC)$ with respect to the intersection form. Since $\cZ_y$ is an elliptic curve, a Torelli marking on $\cZ_y$ is tantamount to the choice of a line spanned by a non-trivial class in $H_1(\cZ_y,\bbC)$. For $\bullet \in \{{\rm LR}, {\rm CF}\}$, each boundary point $y \in \{0,-1/27\}$ selects such a Torelli marking through the associated $\alpha_\bullet$-cycle. 

Let $I=\llbracket 1,k\rrbracket$ and $p_I \in (\cZ_y)^k$ a point on the $k^{\rm th}$-product of a regular fibre. We define 
\beq
\omega_{0,1}^{\bullet}(p) \coloneqq \log z_1 \rd \log z_2\,, \qquad
\omega_{0,2}^{\bullet}(p_1, p_2) \coloneqq B^{\bullet}(p_1,p_2)
\label{eq:bergmann}
\eeq
where $B^{\bullet}(p_1,p_2)$ is the  \emph{Bergman kernel with $\bullet$-Torelli marking} of $\cZ_y$. This is the unique symmetric bi-differential on $\cZ_y^{2}$, with pole divisor supported on the diagonal, and further satisfying, in an affine chart $z: \cZ_y \to \bbA^1$,
\beq
B^\bullet(p_1,p_2) = \frac{\rd z(p_1) \boxtimes \rd z(p_2)}{(z(p_1)-z(p_2))^2} + \mathrm{holomorphic}\,, \qquad 
\int_{\alpha_{\bullet}} B^\bullet(p_1, \cdot) = 0\,.
\eeq
Let $R$ be the ramification locus of $z_2:\cZ_y \to \bbP^1$: 
these are all simple ramification points. 
For $p$ in a small neighbourhood of $p_{\rm cr} \in R$, we write $\iota(p)$ for the image of $p$ under the corresponding deck involution. As in \cref{sec:EO}, we then proceed to define recursively, for $g \geq 0$ and $k>0$ with $2g-2+k>0$, symmetric multi-differentials $\omega_{g,k}^\bullet$ through the Eynard--Orantin topological recursion:

\begin{align*}
\omega^\bullet_{g,k+1}(p_0, p_I)= &
\sum_{p'\in R} \Res_{p=p'} {K}^\bullet(p_0,p)\Bigg[ \omega^\bullet_{g-1,k+2}(p, \iota(p), p_I) \\ 
& +\sum_{\substack{J \subseteq I, 0 \leq h \leq g\,,\\ (J,h) \neq (\emptyset,0),(I,g)}} \omega^\bullet_{h,|J|+1}(p, p_J) \omega^\bullet_{g-h,k-|J|+1}(\iota(p), p_{I \setminus J})\Bigg]\,, \\
\end{align*} 
with
\beq
K^\bullet(p_0,p) \coloneqq \frac{1}{2}\l(\frac{\int_{p}^{\iota(p)} \omega^\bullet_{0,2}(p_0,\cdot )}{\omega^\bullet_{0,1}(p)-\omega^\bullet_{0,1}(\iota(p))}\r)\,. 
\label{eq:intker}
\eeq

The ``remodelling'' approach to higher genus mirror symmetry \cite{Bouchard:2007ys,Eynard:2012nj,Fang:2016svw} then implies that the higher genus large-volume GW generating series of $X$ coincide, up to normalisation, with the free energies of the topological recursion.

\begin{thm}[\cite{Fang:2016svw}]
For $2g-2>0$, we have
    \[
\mathrm{GW}_g^{\rm LR}\bigg|_{Q=\re^{2\pi\ri \Pi_{\alpha_{\rm LR}}}} =  \frac{(-1)^{g-1}}{2-2g} \sum_{p'\in R} \Res_{p=p'}\l[  \l(\int^p \omega^{\rm LR}_{0,1}\r){\omega}^{\rm LR}_{g,1}(p)\r]\,,
\]
\label{thm:bkmp}
\end{thm}

We now {\it define} the genus-$g$ conifold GW potential as
\[
\mathrm{GW}^{\rm CF}_g \coloneqq  \frac{(-1)^{g-1}}{2-2g} \sum_{p'\in R} \Res_{p=p'}\l[  \l(\int^p \omega^{\rm CF}_{0,1}\r){\omega}^{\rm CF}_{g,1}(p)\r]\Bigg|_{y=y(t_{{\rm CF}})}\,.
\]
The change of Torelli marking from the large radius to the conifold polarisation corresponds \cite{Coates:2018hms,Alim:2013eja} to the Fricke involution \eqref{eq:taucf}. The Eynard--Orantin recursion then relates \cite{Eynard:2007hf,Eynard:2007kz} the large radius and conifold GW potentials as
\beq
\mathrm{GW}^{\rm CF}_g =\mathrm{GW}^{\rm LR}_g\Big|_{\tau_{\rm LR} \to -\frac{1}{3\tau_{\rm LR}}=\tau_{\rm CF}}\,. 
\label{eq:GWCFLR}
\eeq

\subsection{Proof of the Conifold Gap Conjecture}

\subsubsection{Comparison of disc potentials}

Upon identifying $\tau=\tau_{\rm CF}$, the logarithm of the canonical coordinate on the $x$-plane $\log x(u)$ and the logarithmic derivative of the $\mu_3$-equivariant planar resolvent $\RR_{0,1}(x(u))$ lift to multi-valued functions on the universal family $\cZ$, with logarithmic monodromies around the three sections $\mathsf{p}_{- 1/6}$, $\mathsf{p}_{1/6}$ and $\mathsf{p}_{0}$. The next Proposition shows how they are related to the logarithms of the two Cartesian projections on the fibres of the Hori--Iqbal--Vafa family, $z_1$ and $z_2$.
\begin{prop}
We have
\begin{align}
\log z_1  = -\ri \sqrt{3}  
\RR_{0,1} +2 \log x -\log(1+x)
-\frac{1}{3}\log(y)
\,, \nn \\
\log z_2 = 
-\ri \sqrt{3}  
\RR_{0,1} - \log x-\log(1+x)
-\frac{1}{3}
\log(y)
\,.
\label{eq:z1z2rx}
\end{align}
In particular,
\[
\rd \RR_{0,1} \wedge \rd\log x = \frac{\ri}{3 \sqrt{3}}\, \rd \log z_1 \wedge \rd \log z_2\,.
\]
\end{prop}

\begin{proof}
We start by defining
\[
H_{(a,b,c)}(u) \coloneqq 
\frac{\ri \re^{3\pi\ri a} \vartheta_1(u-a)^2}{\vartheta_1(u-b)\vartheta_1(u-c)}\,.
\]
Specialising to $\{a,b,c\}=\{-\frac{1}{6},\frac{1}{6},\frac{1}{2}\}$ it is straightforward to check, using \eqref{eq:thetaper}, that $H_{(a,b,c)}$ is meromorphic and doubly-periodic on $\Delta$, with at most simple poles at $u \in \{-1/6, 1/6, 1/2\}$,
\[
H_{(a,b,c)} \in \Gamma(T, \cO_T(\mathsf{p}_{-1/6}+\mathsf{p}_{1/6}+\mathsf{p}_{1/2}))
\,,
\]
where furthermore the product over cyclic permutations of the argument satisfies
\beq
\prod_{a \neq b, a \neq c, b > c} H_{(a,b,c)}(u) = -1
\,.
\label{eq:prodH}
\eeq
Consider now
\[
\mathfrak{P}(u) \coloneqq
\sum_{\{a,b,c\} = \{-\frac{1}{6},\frac{1}{2},\frac{1}{6}\}} H_{(a,b,c)}(u)
\,.
\]
%
It is elementary to verify that the residues of the summands at the poles cancel pairwise:
\beq
\Res_{u \in \{-\frac{1}{6},\frac{1}{6},\frac{1}{2}\}} \mathfrak{P}(u)\rd u =0\,, 
\label{eq:ResP}
\eeq
implying that $\mathfrak{P}(u)$ is constant on $\Delta$, $$\mathfrak{P}(u)=\eta \in \bbC\,.$$
From \eqref{eq:prodH}, in terms of
\[
z_1 \coloneqq -\eta H_{(-\frac{1}{6},\frac{1}{6},\frac{1}{2})}\,, \quad z_2 \coloneqq -\eta H_{(\frac{1}{6},-\frac{1}{6},\frac{1}{2})}\,, \quad z_1 z_2 =  -\frac{\eta^2}{ H_{(\frac{1}{2},-\frac{1}{6},\frac{1}{6})}}\,,
\]
the regularity condition \eqref{eq:ResP} reads 
\[
1+z_1 + z_2 + \eta^{-3} \frac{1}{z_1 z_2} =0 \,.
\]
Setting $y \coloneqq \eta^{-3}$, this coincides with the 
universal curve for the modular group $\Gamma_1(3)$ in \eqref{eq:HV}. Since, by \eqref{eq:xu} and \eqref{eq:r01sol}, we have
\[
H_{(-\frac{1}{6},\frac{1}{6},\frac{1}{2})} = \exp(-\ri \sqrt{3}  
\RR_{0,1}(u))\frac{ x(u)^2}{1+x(u)}\,, \qquad H_{(\frac{1}{6},-\frac{1}{6},\frac{1}{2})} = \frac{\exp(-\ri \sqrt{3}  
\RR_{0,1}(u))}{x(u)(1+x(u))}\,,
\]
we find that
\begin{align*}
\log z_1 = & \log\l(-\eta H_{(-\frac{1}{6},\frac{1}{6},\frac{1}{2})}\r) = -\ri \sqrt{3}  
\RR_{0,1} +2 \log x -\log(1+x)
-\frac{1}{3}\log(y)\,, \\
\log z_2 = & \log\l(-\eta H_{(\frac{1}{6},-\frac{1}{6},\frac{1}{2})}\r) = 
-\ri \sqrt{3}  
\RR_{0,1} - \log x-\log(1+x)
-\frac{1}{3}
\log(y)\,,
\end{align*}
from which the claim follows.
\end{proof}

From the Proposition, we deduce in particular that the periods of $\cS_\chi \omega^{(\Phi, \Lambda)}_{0,1} = \RR_{0,1} \rd \log x$ and $\log z_1 \rd \log z_2$ on the fibres of $\cZ$ coincide\footnote{A slight subtlety here is that, because of log-singularities, the equality assumes  a consistent determination of the logarithm and, dually, a consistent choice of representative for $\gamma$  on both sides. The addition to $\gamma$ of a homologically trivial cycle with non-zero support on $z_2 =0$ or $\infty$ is sensitive to a branch choice for the logarithm, as the ensuing shift of $\log z_1 \rd \log z_2$ by integer multiples of $2\pi\ri \rd \log z_2$ affects periods by the corresponding residue.} up to a constant normalisation:
\[
\int_\gamma \cS_\chi \omega^{(\Phi, \Lambda)}_{0,1} =
\int_\gamma \RR_{0,1} \frac{\rd x}{x} = \frac{\ri}{3 \sqrt{3}}\, \int_\gamma \log z_1 \rd \log z_2
\]
for all $\gamma \in H_1(\cZ_y,\bbC)$.

\subsubsection{Comparison of Torelli markings}

The Eynard--Orantin topological recursion for the higher genus free energies in the conifold model assumes the canonical Torelli marking associated to the line spanned by $\alpha=[\cA]\in H_1(T,\bbC)$. The analogous recursion for the higher genus GW potentials in the conifold polarisation, on the other hand, selects the Torelli marking associated to the line spanned by $\alpha_{\rm CF}\in H_1(\cZ_y,\bbC)$. To show that the two Lagrangians agree, it suffices to show, by the discussion in \cref{sec:HV}, that
\[
\int_\alpha \cS_\chi \omega^{(\Phi, \Lambda)}_{0,1}
= \int_{\alpha} \RR_{0,1} \frac{\rd x}{x} = \cO(1+27 y)
\] 
is holomorphic and linearly vanishing at the conifold point, $y=-1/27$. To see this, note that, from \eqref{eq:taucf}, we have
\[
\mathsf{q} = \mathsf{q}_{\rm CF} \coloneqq \re^{2\pi\ri \tau_{\rm CF}} =\left(y+\frac{1}{27}\right)+\cO\left(1+27 y\right)^2
\]
and using \eqref{eq:thetadef}, \eqref{eq:xu} and \eqref{eq:r01} we find
\[
t = \frac{\ri}{2\pi} \int_{-1/2}^{1/2} \RR_{0,1}(x(u)) \frac{x'(u)}{x(u)} \rd u = \frac{1}{9} \left(1+27 y\right)+\cO\left(1+27 y\right)^2
\,,
\]
from which we in particular deduce that
\[
t = \frac{1}{9}\, t_{\rm CF}\,, \qquad \bbC\bra \alpha \ket = \bbC\bra \alpha_{\rm CF} \ket\,. 
\]

\subsubsection{Comparison of annulus potentials}
Recall, from  \eqref{eq:bergmann}, that $\omega^{\rm CF}_{0,2}=B^{\rm CF}$ is the unique meromorphic bi-differential on $T^{(2)}$ with a double-pole on the diagonal and vanishing period along the Lagrangian \[\mathrm{span}_\bbC \{\alpha_{\rm CF}\} = 
\mathrm{span}_\bbC \{\alpha\}
 \subset H_1(T, \bbC)\,.\] On the other hand, from \eqref{eq:2pwp}--\eqref{eq:2ptorelli} we have that
%
$\cS_\chi \omega_{0,2}^{(\Phi, \Lambda)}$
has the same pole structure on the diagonal, normalisation, and vanishing period along the ${\rm CF}$-Torelli marking:
\[
\int_\alpha \cS_\chi \omega_{0,2}^{(\Phi, \Lambda)}(v, u) = \rd v \int_\alpha \cR_{0,2}(v, u) \rd u = 0
\]
Therefore,
\[
\cS_\chi \omega_{0,2}^{(\Phi, \Lambda)} = \omega^{\rm CF}_{0,2}\,.
\]

\subsubsection{Symplectic invariance and the Conifold Gap Conjecture}

Consider the following four basic automorphisms of $(\bbC^\times)^2$:
\begin{align*}
    \mathsf{T}_{f}\,:\, & (z_1, z_2) \to  (z_1 f(z_2), z_2)\,, \\
    \mathsf{S}\,:\, & (z_1, z_2) \to 
     (-1/z_2, z_1)\,, \\
    \mathsf{E}_{(c_1,c_2)} \,:\, & (z_1, z_2) \to(c_1 z_1 , c_2 z_2)\,, \\
    \mathsf{D}_{(w_1,w_2)} \,:\, & (\log z_1, \log z_2) \to (w_1 \log z_1, w_2 \log z_2)\,,
\end{align*}
where $f(z)$ is a meromorphic function of its argument, and $w_i, c_i \in \bbC$, $i=1,2$. The first three are symplectic automorphisms of the 2-torus w.r.t. the canonical symplectic form $\rd \log z_1 \wedge \rd \log z_2$, while $\mathsf{D}_{w_1,w_2}$ is a global scaling transformation with constant factor $w_1 w_2$.  

Let $P\in \bbC[z_1^{\pm 1},z_2^{\pm 1}]$, and $C_P$ be the smooth compactification of its vanishing locus in the 2-torus,
\[
C_P = \overline{\l\{(z_1,z_2)\in (\bbC^\times)^2 \Big|\, P(z_1,z_2)=0\r\}}\,.
\]
We say that the triple $(C_P, z_1, z_2)$ given by the curve $C_P$ with the two marked meromorphic Cartesian projections $(z_1, z_2)$ is \emph{regularly framed} if, for $p\in C_P$,
\[
z_1(p) \in \{0,\infty\} \Longleftrightarrow z_2(p) \in \{0,\infty\}\,.
\]
In other words, $(C_P, z_1, z_2)$ is regularly framed if the divisor of (simple) poles of $\rd \log z_1$ has the same support on $C_P$ as that of $\rd \log z_2$. As such, this is a feature of the specific embedding into the torus (and therefore of the polynomial $P$) rather than an intrinsic property of the curve $C_P$: it can be lost, or acquired, by replacing $P \rightarrow g^* P$ for an automorphism $g \in \mathrm{Aut}((\bbC^\times)^2) \simeq \mathrm{GL}(2,\bbZ)$.

Fix now a Torelli marking on $C_P$, $g \in \bbZ_{\geq 2}$, and let $F_g$ be the genus-$g$ topological recursion free energies on $C_P$ with initial data 
\[
\omega_{0,1} \coloneqq \log  z_1 \rd \log z_2\,, \qquad \omega_{0,2} \coloneqq B(z_1,z_2)\,,
\]
with $B(z_1, z_2)$ the Bergman kernel associated to the marked Torelli Lagran\-gian in $H_1(C_P,\bbC)$. For \[\theta \in \{\mathsf{T}_f, \mathsf{S}, \mathsf{E}_{(c_1,c_2)},\mathsf{D}_{(w_1,w_2)}\} \subset \mathrm{GL}_2(\bbC) \ltimes \mathrm{Symp}((\bbC^\times)^2, \rd\log z_1 \wedge \rd\log z_2)\,,\] we denote by
\[
\theta[\omega_{g,n}]\,, \qquad \theta[F_g]
\]
the Eynard--Orantin correlators and free energies obtained from the topological recursion with initial data
\[
\theta[\omega_{0,1}] \coloneqq \theta^* \omega_{0,1}
\,, \qquad \theta[\omega_{0,2}] \coloneqq \omega_{0,2}\,. 
\]
When $\theta=\mathsf{T}_f$ or $\theta =\mathsf{E}_{(c_1,c_2)}$, it is immediate to check \cite{Eynard:2007hf,Eynard:2009phf} that the topological recursion kernel \eqref{eq:intker} is $\theta$-invariant, and as a consequence so are the higher genus free energies:
\[
\mathsf{T}_f[F_g]=
\mathsf{E}_{(c_1,c_2)}[F_g] = 
F_g\,.
\]
Likewise, for the $\mathsf{D}_{(w_1,w_2)}$-transformed initial data
\[
 \mathsf{D}_{(w_1,w_2)}[\omega_{0,1}] \coloneqq w_1 w_2 \log z_1 \rd\log z_2 \,, \qquad  \mathsf{D}_{(w_1,w_2)}[\omega_{0,2}] \coloneqq B(z_1,z_2)\,,
\]
the recursion kernel gets scaled by $(w_1 w_2)^{-1}$, and
the corresponding free energies are related as \cite{Eynard:2009phf}
\[
\mathsf{D}_{(w_1,w_2)}[F_g] = (w_1 w_2)^{2-2g }F_g\,.
\]
Finally, for $\theta = \mathsf{S}$, the free energies $\mathsf{S}[F_g]$ and $F_g$ are generally related in a complicated way \cite[Thm.~3.2]{Hock:2025wlm}, involving corrections by sums of residues at poles of $\rd \log z_1$ which are not simultaneously poles of $\rd \log z_2$ (and viceversa). If however $(C_P, z_1, z_2)$ is regularly framed these corrections vanish, and
\[
\mathsf{S}[F_g] = F_g
\,.
\]

\begin{thm}
    For $g\geq 2$, we have 
    \[
    F^{(\Phi, \Lambda)}_g = 27^{g-1}
        \mathrm{GW}^{\rm CF}_g\Big|_{y=y(t_{\rm CF})}\,.
    \]
    up to a constant independent of $t=t_{\rm CF}/9$.
\label{thm:mm=gw}
\end{thm}
\begin{proof}

The change-of-variables \eqref{eq:z1z2rx} can be written as the composition
\beq
\mathsf{D}_{(\ri \sqrt{3},1)} \circ
\mathsf{S} \circ \mathsf{E}_{(y^{1/3},y^{1/3})}\circ \mathsf{T}_{1/z_2} \circ \mathsf{S} \circ \mathsf{D}_{(-1,-3)} \circ \mathsf{T}_{1/z_2}\,.
\label{eq:word}
\eeq
It is immediate to check that the
Hori--Iqbal--Vafa defining polynomial in \eqref{eq:HV} is regularly framed, as are its successive compositions with each factor in the word \eqref{eq:word}. Using 
\cref{prop:symmFg,prop:EOMM,thm:bkmp} we therefore obtain that
\[
F^{(\Phi,\Lambda)}_g\bigg|_{t=t_{\rm CF}/9} = 27^{g-1}\, \mathrm{GW}_g^{\rm CF}\,,
\]
up to a $t$-independent additive constant, concluding the proof.
\end{proof}
Expanding at $t=t_{\rm CF}=0$ using \cref{prop:HZ1}, we deduce the statement of the Conifold Gap Conjecture for local $\bbP^2$.
\begin{cor}
For all $g \geq 2$, we have
\[
\mathrm{GW}_g^{\rm CF} = \frac{3^{g-1} B_{2g}}{2g(2g-2)} t_{\rm CF}^{2-2g}  +\cO(1)\,. 
\]
\label{cor:cg}
\end{cor}

\crefalias{section}{appendix}
\begin{appendix}
    \section{Proof of \cref{lem:phichi}.}
\label{sec:phichi}

The Weierstrass function is defined by
\beq
\wp(u) \coloneqq \frac{1}{u^2} + \sum_{m \in M \setminus 0} \l[ \frac{1}{(u-m)^2}-\frac{1}{m^2} \r] \,,
\label{eq:wp}
\eeq
in terms of which $T$ is the compactification of the affine cubic
\[
\big\{ (W,Z) \in \bbC^2\,\big|\,
W^2 = 4 Z^3-g_2(\mathsf{q}) Z- g_3(\mathsf{q}) \big\}
\]
via $W \to \wp'(u)$, $Z \to \wp(u)$. As $M$ is invariant under complex conjugation, the  Weierstrass invariants $g_2(\mathsf{q})$ and $g_3(\mathsf{q})$ are real, and therefore $\wp(v) \in \bbR$ for $v \in \bbR$. Moreover, the lattice roots
\[
e_i \coloneqq \wp(\omega_i)\,,\qquad i=1, 2, 3\,,\]
are all real and distinct in this case \cite[Sec.~20.32]{MR1424469}. 

We start by analysing $\mathrm{Im}(\phi_\pm(u))$. Recalling  that
\[
\wp(u) = -\de_u^2\log\vartheta_1(u) + \mathrm{const},
\]
we can write
\[
\phi_\pm''(u) = \wp\l(u-\frac{1}{6}\pm \frac{\tau}{2}\r)-\wp\l(u+\frac{1}{6}\pm \frac{\tau}{2}\r)\,.
\]
 From the general addition formula for the Weierstrass function under translation by half-periods,
\[
\wp \left(u\pm \omega _i\right)=\frac{\left(e_i-e_j\right) \left(e_i-e_k\right)}{\wp \left(u\right)-e_i}+e_i\,, \qquad i\neq j\neq k\,,
\]
we deduce that 
\beq
\mathrm{Im}\,\wp\l(v\pm \frac{\tau}{2}
\r) =0
\label{eq:wpreal}
\eeq
for $v \in \bbR$. In particular, by parity and Schwarz-reflection symmetry of $\wp(u)$,
$\phi_\pm''(v)$ is real and odd for $v \in (-1/2,1/2]$. Integrating \eqref{eq:wpreal} once w.r.t. $v$, we find that
\[
\mathrm{Im}\l[
\frac{\vartheta_1'\l(v\pm \frac{\tau}{2}\r)}{\vartheta_1\l(v\pm \frac{\tau}{2}\r)}\r]=\mathrm{const}
\]
for $v\in (-1/2,1/2]$, from which we deduce that the difference
\[
\phi_\pm'(v)=\frac{\vartheta_1'\l(v+\frac{1}{6}\pm \frac{\tau}{2}\r)}{\vartheta_1\l(v+\frac{1}{6}\pm \frac{\tau}{2}\r)}-\frac{\vartheta_1'\l(v-\frac{1}{6}\pm \frac{\tau}{2}\r)}{\vartheta_1\l(v-\frac{1}{6}\pm \frac{\tau}{2}\r)}
\]
is even and real for $v \in \bbR$, and  therefore
$\mathrm{Im}(\phi_\pm(v))$
is constant throughout $(-1/2,1/2]$. Evaluating at $v=0$, and using \eqref{eq:thetaper}, we find
\beq
\phi_\pm(0) = 
\log \frac{\vartheta_1\l(\frac{1}{6}\pm\frac{\tau}{2}\r)}{\vartheta_1\l(\frac{1}{6}\mp\frac{\tau}{2}\r)}
=
\log(-\re^{\mp \pi \ri/3})
=\pm \frac{2\pi\ri}{3} = \ri\, \mathrm{Im}(\phi_\pm(v))\,.
\label{eq:phi0}
\eeq
Consider now $\mathrm{Re}(\phi_\pm(v))$. We have just shown that it vanishes at $v=0$, and since $\phi_\pm'(v)$ is real and even, it is an odd function of $v$. It is also straightforward to check, by the same quasi-periodicity argument, that $v=1/2$ is also a zero. 
We shall show that it has exactly one maximum and one minimum, with opposite values, for $v \in (-1/2,1/2]$. To see this, 
using \eqref{eq:thetadef}, we can check that 
\[
\phi_\pm'(0)>0 \quad \mathrm{and} \quad \phi_\pm'(\pm 1/2)<0 \quad \mathrm{for}~~ \mathrm{Im}(\tau) \gg 1
\,\, (|\mathsf{q}| \ll 1)\,,
\]
meaning that $\phi_\pm(v)$ has at least two zeroes in the interval $(-1/2,1/2)$. By \eqref{eq:thetaper}, $\phi_\pm'(u)$ is meromorphic and doubly-periodic along $M$ with two simple poles in its fundamental domain: therefore it has {\it exactly} two zeroes $\pm \mathfrak{u} \in (-1/2,1/2)$, with $\mathfrak{u}>0$, corresponding to the critical points of $\mathrm{Re}(\phi_\pm(v))$ in the same interval. 
As $\mathrm{Re}(\phi_\pm(v))$ is odd and vanishing at $v=\pm 1/2$, these must be a maximum and a minimum, with
\[
\mathfrak{x} \coloneqq 
\mathrm{Re}(\phi_\pm(\mathfrak{u}))=-\mathrm{Re}(\phi_\pm(-\mathfrak{u}))\,.
\]
%
%
%
Consider now $\chi_+(u)$. We have
\[
\chi_+''(u) = \tau^2 \l(\wp\l(\tau u+\frac{1}{3}\r)-\wp\l(\tau u+\frac{2}{3}\r)\r)\,,
\]
and in particular $\chi_+''(u)$ has four zeroes in $\Delta$ counted with multiplicity, having double poles at $u=-1/(3\tau)$ and $u=-2/(3\tau)$. It is elementary to check, using parity and periodicity of $\wp(u)$, that these are located at
\[
u=0\,, \qquad u= \frac{1}{2}\,, \qquad u=\pm \frac{1}{2\tau}\,.
\]
Using the reality property $\overline{\wp}(u)=\wp(\overline{u})$ 
we further have, for $v \in (-1/2, 1/2]$, that
\begin{align*}
\mathrm{Re} \chi_+''(v) = & \frac{\tau^2}{2} \l[\wp\l(\tau\, v+\frac{1}{3}\r)-\wp\l(\tau\, v+\frac{2}{3}\r)+\overline{\wp}\l(\tau\, v+\frac{1}{3}\r)-\overline{\wp}\l(\tau\, v+\frac{2}{3}\r)\r] =0\,,
\end{align*}
hence
\[
\mathrm{Re} \chi_+'(v) = \mathrm{const}\,.
\]
Evaluating at $v=0$, we find
\[
\mathrm{Re} \chi_+'(0) = |\tau| \mathrm{Im} \frac{x'(1/2)}{x(1/2)} = |\tau| \mathrm{Im} \l[
\frac{\vartheta_1'}{\vartheta_1}\l(\frac{2}{3}\r)-
\frac{\vartheta_1'}{\vartheta_1}\l(\frac{1}{3}\r)
\r] = 0\,,
\]
hence
\[
\mathrm{Re} \chi_+(v) = \mathrm{const} = \mathrm{Re} \chi_+(0) = \log(x(1/2))=0 \,.
\]

As far as $\mathrm{Im}(\chi_+(v)) = -\ri \chi_+(v)$ is concerned, the foregoing discussion implies that $\mathrm{Im}(\chi_+''(v))$ is non-constant, odd, and vanishing only at $v=\pm 1/2$ and $v=0$: hence it has strictly negative (resp. positive) values for $v<0$ (resp. $v>0$), 
with exactly one maximum and minimum having opposite signs. As a consequence, $\mathrm{Im}(\chi_+'(v))$ is even and with exactly one minimum at the origin, and a maximum at $v=\pm \frac{1}{2}$, where it is readily checked to be negative, hence 
\[
\mathrm{Im}(\chi_+'(v))<0 \quad  \mathrm{for}~~ -1/2<v\leq 1/2\,.\]
As we took the branch cut of the logarithm in the definition of $\chi_+(v)$ to run along the negative real axis, $\mathrm{Im}(\chi_+(v))$
will have a $2\pi$-jump discontinuity at $v=0$, and is otherwise continuous and differentiable for $v \in (-1/2,0)$ (resp. $v \in (0,1/2]$), where it is monotonically strictly decreasing, and negative and bounded above by $\mathrm{Im}(\chi_+(-1/2))=-2\pi/3$ (resp. positive and bounded below by $\mathrm{Im}(\chi_+(1/2))=2\pi/3$). A \emph{verbatim} application of the same arguments to $\chi_-(u)$ concludes the proof.

\end{appendix}
\bibliography{refs}
	

\end{document}